\documentclass{amsart}
\usepackage{mathrsfs}
\usepackage{mathrsfs}
\usepackage{mathrsfs}
\usepackage{mathrsfs}
\usepackage{amsfonts}
\usepackage{cases}
\usepackage{latexsym}
\usepackage{amsmath}
\usepackage[arrow,matrix]{xy}
\usepackage{stmaryrd}
\usepackage{amsfonts}
\usepackage{amsmath,amssymb,amscd,bbm,amsthm,mathrsfs,dsfont}
\usepackage{fancyhdr}
\usepackage{amsxtra,ifthen}
\usepackage{verbatim}

\numberwithin{equation}{section}

\theoremstyle{plain}
\newtheorem{theorem}{Theorem}[section]
\newtheorem{lemma}[theorem]{Lemma}
\newtheorem{proposition}[theorem]{Proposition}
\newtheorem{corollary}[theorem]{Corollary}
\newtheorem{conjecture}[theorem]{Conjecture}

\theoremstyle{definition}

\newtheorem{question}[theorem]{Question}

\begin{document}

\title[Centralizing traces and Lie triple isomorphisms]
{Centralizing traces and Lie triple isomorphisms on generalized matrix algebras}

\author{Ajda Fo\v sner, Xinfeng Liang, Feng Wei and Zhankui Xiao}

\address{Fo\v sner: Faculty of Management, University of Primorska,
Cankarjeva 5, SI-6104 Koper, Slovenia}

\email{ajda.fosner@fm-kp.si}

\address{Liang: School of Mathematics and Statistics, Beijing Institute of
Technology, Beijing, 100081, P. R. China}

\email{lxfrd@163.com}

\address{Wei: School of Mathematics and Statistics, Beijing Institute of
Technology, Beijing, 100081, P. R. China}

\email{daoshuo@hotmail.com}\email{daoshuowei@gmail.com}

\address{Xiao: School of Mathematical Sciences, Huaqiao University,
Quanzhou, Fujian, 362021, P. R. China}

\email{zhkxiao@gmail.com}\email{zhkxiao@bit.edu.cn}

\begin{abstract}
Let $\mathcal{G}$ be a generalized matrix algebra over a commutative
ring $\mathcal{R}$ and $\mathcal{Z(G)}$ be the center of $\mathcal{G}$. Suppose that ${\mathfrak q}\colon \mathcal{G}\times
\mathcal{G}\longrightarrow \mathcal{G}$ is an $\mathcal{R}$-bilinear
mapping and ${\mathfrak T}_{\mathfrak q}\colon
\mathcal{G}\longrightarrow \mathcal{G}$ is the trace of
$\mathfrak{q}$. We describe the form of ${\mathfrak T}_{\mathfrak
q}$ satisfying the condition $[{\mathfrak T}_{\mathfrak
q}(G), G]\in \mathcal{Z(G)}$ for all $G\in \mathcal{G}$.
The question of when ${\mathfrak T}_{\mathfrak q}$ has the proper
form is considered. Using the aforementioned trace function, we
establish sufficient conditions for each Lie triple isomorphism of
$\mathcal{G}$ to be almost standard. As applications we characterize
Lie triple isomorphisms of full matrix algebras, of triangular algebras and
of certain unital algebras with nontrivial idempotents. Some topics for future research closely related to our current work are proposed at the end of
this article.
\end{abstract}



\thanks{The work of the fourth author is supported by the National Natural Science Foundation of China
(Grant No. 11226068 and No. 11301195).}

\maketitle

\section{Introduction}\label{xxsec1}

Let $\mathcal{R}$ be a commutative ring with identity, $\mathcal{A}$
be a unital algebra over $\mathcal{R}$ and $\mathcal{Z(A)}$ be the
center of $\mathcal{A}$. Let us denote the commutator or the Lie
product of the elements $a, b\in \mathcal{A}$ by $[a, b]=ab-ba$.
Recall that an $\mathcal{R}$-linear mapping ${\mathfrak f}:
\mathcal{A}\longrightarrow \mathcal{A}$ is said to be
\textit{centralizing} if $[{\mathfrak f}(a), a]\in \mathcal{Z(A)}$ for all $a\in
\mathcal{A}$. In particular, the mapping $\mathfrak{f}$ is called \textit{commuting} if $[{\mathfrak f}(a), a]=0$ for all $a\in \mathcal{A}$. When we investigate a centralizing (or commuting) mapping, the principal
task is to describe its form. The identity mapping and every mapping
which has its range in $\mathcal{Z(A)}$ are two classical examples
of commuting mappings. Furthermore, the sum and the pointwise
product of commuting mappings are also commuting mappings. For instance, the mapping
$$
\mathfrak{f}(x)=\lambda_0(x)x^n+\lambda_1(x)x^{n-1}+\cdots\cdots+\lambda_{n-1}(x)x+\lambda_n(x), \hspace{6pt} \lambda_i\colon \mathcal{A}\longrightarrow \mathcal{Z(A)} \eqno(1.1)
$$ 
is commuting for any choice of central maps $\lambda_i$. Of course, there are other examples; namely, elements commuting with $x$ may not necessarily be equal to a polynomial in $x$ (with central coefficients) and so in most rings there is a variety of possibilities
of how to find commuting maps different from those described in $(1.1)$. We
encourage the reader to read the well-written survey paper
\cite{Bresar3}, in which the author presented the development of the
theory of commuting mappings and their applications in details.

Let $n$ be a positive integer and $\mathfrak{q}\colon
\mathcal{A}^n\longrightarrow \mathcal{A}$. We say that
$\mathfrak{q}$ is $n$-\textit{linear} if ${\mathfrak q}(a_1,\cdots,
a_n)$ is $\mathcal{R}$-linear in each variable $a_i$, that is,
${\mathfrak q}(a_1, \cdots, ra_i+sb_i, \cdots, a_n)=r {\mathfrak
q}(a_1, \cdots, a_i, \cdots, a_n)+s{\mathfrak q}(a_1, \cdots, b_i,
\cdots, a_n)$ for all $r, s\in \mathcal{R}, a_i, b_i\in \mathcal{A}$
and $i=1, 2, \cdots, n$. The mapping ${\mathfrak T}_{\mathfrak
q}\colon \mathcal{A}\longrightarrow \mathcal{A}$ defined by
${\mathfrak T}_{\mathfrak q}(a)={\mathfrak q}(a, a, \cdots, a)$ is
called a \textit{trace} of ${\mathfrak q}$. We say that a centralizing
trace ${\mathfrak T}_{\mathfrak q}$ is \textit{proper} if it can be written as
$$
{\mathfrak T}_{\mathfrak q}(a)=\sum_{i=0}^{n}\mu_i(a)a^{n-i}, 
\hspace{10pt}
\forall a\in \mathcal{A},
$$
where $\mu_i(0\leq i\leq n)$ is a mapping from $\mathcal{A}$ into
$\mathcal{Z(A)}$ and every $\mu_i(0\leq i\leq n)$ is in fact a trace
of an $i$-linear mapping ${\mathfrak q}_i$ from $\mathcal{A}^i$ into
$\mathcal{Z(A)}$. Let $n=1$ and ${\mathfrak f}\colon
A\longrightarrow A$ be an $\mathcal{R}$-linear mapping. In this
case, an arbitrary trace ${\mathfrak T}_{\mathfrak f}$ of
${\mathfrak f}$ exactly equals to itself. Moreover, if a
centralizing trace ${\mathfrak T}_{\mathfrak f}$ of ${\mathfrak f}$
is proper, then it has the form
$$
{\mathfrak T}_{\mathfrak f}(a)=z a+\eta(a),  \hspace{10pt} \forall a\in
\mathcal{A},
$$
where $z\in \mathcal{Z}(A)$ and $\eta$ is an $\mathcal{R}$-linear mapping from $\mathcal{A}$ into $\mathcal{Z(A)}$. Let us see the case of $n=2$. Suppose
that ${\mathfrak g}\colon \mathcal{A}\times
\mathcal{A}\longrightarrow \mathcal{A}$ is an $\mathcal{R}$-bilinear
mapping. If a centralizing trace ${\mathfrak T}_{\mathfrak g}$ of
${\mathfrak g}$ is proper, then it is of the form
$$
{\mathfrak T}_{\mathfrak g}(a)=z a^2+\mu(a)a +\nu(a), \hspace{10pt} \forall a\in \mathcal{A},
$$
where $z\in \mathcal{Z(A)}$, $\mu$ is an $\mathcal{R}$-linear
mapping from $\mathcal{A}$ into $\mathcal{Z(A)}$ and $\nu$ as a trace of some bilinear mapping maps $\mathcal{A}$ into $\mathcal{Z(A)}$. Bre\v{s}ar started the study
of commuting and centralizing traces of multilinear mappings in his
series of works \cite{Bresar0, Bresar1, Bresar2, Bresar3}, where he
investigated the structure of centralizing traces of (bi-)linear
mappings on prime rings. It was proved that in certain rings, in
particular, prime rings of characteristic different from $2$ and
$3$, every centralizing trace of a biadditive mapping is commuting.
Moreover, every centralizing mapping of a prime ring of
characteristic not $2$ is of the proper form and is actually
commuting. An exciting discovery \cite{FosnerWeiXiao} is that every
centralizing trace of arbitrary bilinear mapping on triangular
algebras is commuting in some additional conditions.
It has turned out that this
study is closely related to the problem of characterizing Lie
isomorphisms or Lie derivations of associative rings
\cite{BeidarMartindaleMikhalev, BenkovicEremita1, BresarChebotarMartindale, BresarEremitaVillena, FosnerWeiXiao, XiaoWei2}. 
Lee et al further generalized Bre\v{s}ar's results by showing that each commuting trace of an arbitrary multilinear
mapping on a prime ring also has the proper form
\cite{LeeWongLinWang}. On the other hand, centralizing mappings are looked on as the most basic and important examples of functional identities \cite{BresarChebotarMartindale}.

Cheung in \cite{Cheung2} studied commuting mappings of triangular
algebras (e.g., of upper triangular matrix algebras and nest
algebras). He determined the class of triangular algebras for which
every commuting mapping is proper. Xiao and Wei \cite{XiaoWei1}
extended Cheung's result to the case of generalized matrix algebras.
They established sufficient conditions for each commuting mapping of
a generalized matrix algebra $
\left[\smallmatrix A & M\\
N & B \endsmallmatrix \right]$ to be proper. Benkovi\v{c} and Eremita
\cite{BenkovicEremita1} considered commuting traces of bilinear
mappings on a triangular algebra $
\left[\smallmatrix A & M\\
O & B \endsmallmatrix \right]$. They gave conditions under which
every commuting trace of a triangular algebra 
$
\left[\smallmatrix A & M\\
O & B \endsmallmatrix \right]$ 
is proper. More recently, Franca \cite{Franca1, Franca2} investigated commuting mappings on subsets of matrices that are not closed under addition such as invertible matrices, singular matrices, matrices of rank-$k$, etc. The research results demonstrate that the commuting mappings on these sets basically have the proper form. Liu \cite{Liu1, Liu2} immediately extended
Franca's works to the case of centralizing mappings. These works explicitly imply that functional identities can be developed to the sets that are not closed under addition. The form of commuting traces of multilinear mappings of upper triangular matrix algebras was earlier described in
\cite{BeidarBresarChebotar}. Simultaneously, some researchers engage in characterizing $k$-commuting mappings of generalized matrix algebras and those of unital algebras with notrivial idempotents, see \cite{DuWang1, LiWei2, QiHou3}. Motivated by Benkovi\v{c} and Eremita's work \cite{BenkovicEremita1}, the present authors \cite{FosnerWeiXiao} studied centralizing traces of triangular algebas. It is shown that under some mild conditions every
centralizing trace of an arbitrary bilinear mapping on a triangular
algebra $\mathcal{T}$ is commuting.  It is natural to ask the following question

\begin{question}\cite[Qustion 5.1]{FosnerWeiXiao}\label{xxsec1.1}
Let $\mathcal{G}=\left[\smallmatrix
A & M\\
N & B \endsmallmatrix \right]$ be a generalized matrix algebra over $\mathcal{R}$ and ${\mathfrak q}\colon \mathcal{G}\times
\mathcal{G}\longrightarrow \mathcal{G}$ be an $\mathcal{R}$-bilinear
mapping. Under what conditions, every centralizing trace ${\mathfrak T}_{\mathfrak q}:
\mathcal{G}\longrightarrow \mathcal{G}$ of ${\mathfrak q}$ has
the proper form ?
\end{question} 

\noindent One of the main aims of this article is
to address the above question and provide a sufficient condition for each centralizing trace of
arbitrary bilinear mappings on a generalized matrix algebra $
\left[\smallmatrix A & M\\
N & B \endsmallmatrix \right]$ to be proper. Consequently, this makes
it possible for us to characterize centralizing traces of bilinear
mappings on full matrix algebras, those of bilinear mappings on
triangular algebras and those of bilinear mappings on certain unital
algebras with nontrivial idempotents.

Another important purpose of this article is devoted to the Lie
triple isomorphisms problem of generalized matrix algebras. At his 1961 AMS
Hour Talk, Herstein proposed many problems concerning the structure
of Jordan and Lie mappings in associative simple and prime rings
\cite{Herstein}. The renowned Herstein's Lie-type mapping research
program was formulated since then. The involved Lie mappings mainly
include Lie isomorphisms, Lie triple isomorphisms, Lie derivations
and Lie triple derivations et al. Given a commutative ring
$\mathcal{R}$ with identity and two associative
$\mathcal{R}$-algebras $\mathcal{A}$ and $\mathcal{B}$, one define a
\textit{Lie triple isomorphism} from $\mathcal{A}$ into $\mathcal{B}$ to be
an $\mathcal{R}$-linear bijective mapping ${\mathfrak l}$ satisfying
the condition
$$
{\mathfrak l}([[a, b], c])=[[{\mathfrak l}(a), {\mathfrak l}(b)], \mathfrak{l}(c)]
\hspace{8pt} \forall a, b, c\in \mathcal{A}.
$$
For example, an isomorphism or the negative of an anti-isomorphism of
one algebra onto another is also a Lie isomorphism. Furthermore,
every Lie isomorphism and every Jordan isomorphism are Lie triple isomorphisms. One can ask
whether the converse is true in some special cases. That is, does
every Lie triple isomorphism between certain associative algebras arise
from isomorphisms and anti-isomorphisms in the sense of modulo
mappings whose range is central ? If $\mathfrak{m}$ is an
isomorphism or the negative of an anti-isomorphism from
$\mathcal{A}$ onto $\mathcal{B}$ and $\mathfrak{n}$ is an
$\mathcal{R}$-linear mapping from $\mathcal{A}$ into the center
$\mathcal{Z(B)}$ of $\mathcal{B}$ such that $\mathfrak{n}([[a,b], c])=0$
for all $a, b, c\in \mathcal{A}$, then the mapping
$$
\mathfrak{l}=\mathfrak{m}+\mathfrak{n} \eqno(\spadesuit)
$$
is a Lie triple homomorphism. We shall say that a Lie triple isomorphism
$\mathfrak{l}\colon A\longrightarrow B$ is \textit{standard} in the
case where it can be expressed in the preceding form $(\spadesuit)$.

The resolution of Herstein's Lie isomorphisms problem in matrix
algebra background has been well-known for a long time. Hua
\cite{Hua} proved that every Lie automorphism of the full matrix
algebra $\mathcal{M}_n(\mathcal{D})(n\geq 3)$ over a division ring
$\mathcal{D}$ is of the standard form $(\spadesuit)$. This result
was extended to the nonlinear case by Dolinar \cite{Dolinar2}
and \v{S}emrl \cite{Semrl} and was further refined by them.
Dokovi\'{c} \cite{Dokovic} showed that every Lie automorphism of
upper triangular matrix algebras $\mathcal{T}_n(\mathcal{R})$ over a
commutative ring $\mathcal{R}$ without nontrivial idempotents has
the standard form as well. Marcoux and Sourour
\cite{MarcouxSourour1} classified the linear mappings preserving
commutativity in both directions (i.e., $[x,y] = 0$ if and only if
$[\mathfrak{f}(x), \mathfrak{f}(y)]=0$) on upper triangular matrix
algebras $\mathcal{T}_n(\mathbb{F})$ over a field $\mathbb{F}$. Such
a mapping is either the sum of an algebraic automorphism of
$\mathcal{T}_n(\mathbb{F})$ (which is inner) and a mapping into the
center $\mathbb{F}I$, or the sum of the negative of an algebraic
anti-automorphism and a mapping into the center $\mathbb{F}I$. The
classification of the Lie automorphisms of
$\mathcal{T}_n(\mathbb{F})$ is obtained as a consequence.
Benkovi\v{c} and Eremita \cite{BenkovicEremita1} directly applied
the theory of commuting traces to the study of Lie isomorphisms on a
triangular algebra $
\left[\smallmatrix A & M\\
O & B \endsmallmatrix \right]$. They provided sufficient conditions
under which every commuting trace of $
\left[\smallmatrix A & M\\
O & B \endsmallmatrix \right]$ is proper. This is directly applied
to the study of Lie isomorphisms of $
\left[\smallmatrix A & M\\
O & B \endsmallmatrix \right]$. It turns out that under some mild
assumptions, each Lie isomorphism of $
\left[\smallmatrix A & M\\
O & B \endsmallmatrix \right]$ has the standard form $(\spadesuit)$.
These results are further extended to the case of generalized matrix algebras
\cite{XiaoWei2}. The present authors \cite{FosnerWeiXiao} investigated Lie triple isomorphisms of triangular algebas via centralizing  
trace techniques. We establish sufficient conditions for each Lie
triple isomorphism on $\mathcal{T}$ to be almost standard. 
This naturally gives rise to the following question. 

\begin{question}\cite[Qustion 5.4]{XiaoWei2}\label{xxsec1.2}
Let $\mathcal{G}=\mathcal{G}(A, M, N, B)$ and $\mathcal{G}^\prime=\mathcal{G}^\prime(
A^\prime, M^\prime, N^\prime, B^\prime)$ be generalized matrix algebras over a commutative ring $\mathcal{R}$ with
$1/2\in\mathcal{R}$. Let $\mathfrak{l}\colon
\mathcal{G}\longrightarrow\mathcal{G^\prime}$ be a Lie triple
isomorphism. Under what conditions does $\mathfrak{l}$ has a
decomposition expression similar to $(\spadesuit)$  ?
\end{question}

\noindent{Simultaneously}, Lie triple isomorphisms between rings
and between (non-)self-adjoint operator algebras have received a
fair amount of attention and have also been intensively studied. The
involved rings and operator algebras include (semi-)prime rings, the
algebra of bounded linear operators, $C^\ast$-algebras, von Neumann
algebras, $H^\ast$-algebras, Banach space nest algebras, Hilbert
space nest algebras, reflexive algebras, see \cite{BaiDuHou1,
BaiDuHou2, BanningMathieu, CalderonGonzalez1, CalderonGonzalez2,
CalderonGonzalez3, CalderonHaralampidou, Lu, MarcouxSourour2,
Mathieu, Miers1, Miers2, QiHou1, QiHou2, Semrl, Sourour, WangLu,
YuLu, ZhangZhang}. One recent remarkable work concerning Lie isomorphisms between Bnanach space nest algebras is due to Qi, Hou and Deng \cite{QiHouDeng}. Let $\mathcal{N}$ and $\mathcal{M}$ be nests on Banach spaces $X$ and $Y$ over the (real or complex)
field $\mathbb{F}$ and let ${\rm Alg}\mathcal{N}$ and ${\rm Alg}\mathcal{M}$ be the associated nest algebras, respectively. It is shown that
a mapping $\Phi\colon  {\rm Alg}\mathcal{N}\longrightarrow {\rm Alg}\mathcal{M}$ is a Lie ring isomorphism (i.e., $\Phi$ is additive, Lie multiplicative
and bijective) if and only if $\Phi$ has the form $\Phi=TAT^{-1}+h(A)I$ for all $A\in {\rm Alg}\mathcal{N}$ or $\Phi(A)=−TA^{*}T^{-1}+h(A)I$ for all $A\in {\rm Alg}\mathcal{N}$, where $h$ is an additive functional vanishing
on all commutators and $T$ is an invertible bounded linear or conjugate linear operator when ${\rm dim}X=\infty$; $T$ is a bijective $\tau$-linear transformation for some field automorphism $\tau$ of $\mathbb{F}$ when
${\rm dim}X<\infty$.

This article is organized as follows. Section $2$ contains the definition of generalized matrix
algebra and some classical examples. In Section $3$ we will give a suitable answer to Question \ref{xxsec1.1} and build
sufficient conditions for each centralizing trace of an arbitrary bilinear
mapping on a generalized matrix algebra $
\left[\smallmatrix A & M\\
N & B \endsmallmatrix \right]$ to be proper (Theorem
\ref{xxsec3.4}). And then we apply this result to describe the
commuting traces of various generalized matrix algebras. In Section
$4$ we will treat Question \ref{xxsec1.2} and give a sufficient condition under which every Lie triple
isomorphism from a generalized matrix algebra into another one has
the almost standard form (Theorem \ref{xxsec4.3}). As corollaries of
Theorem \ref{xxsec4.3}, characterizations of Lie triple isomorphisms on
triangular algebras, on full matrix algebras and on certain unital
algebras with nontrivial idempotents are obtained. The last section
presents some potential further research topics related to our
current work.

\section{Generalized Matrix Algebras and Examples}\label{xxsec2}

Let us begin with the definition of generalized matrix algebras
given by a Morita context. Let $\mathcal{R}$ be a commutative ring
with identity. A \textit{Morita context} consists of two unital
$\mathcal{R}$-algebras $A$ and $B$, two bimodules $_AM_B$ and
$_BN_A$, and two bimodule homomorphisms called the pairings
$\Phi_{MN}: M\underset {B}{\otimes} N\longrightarrow A$ and
$\Psi_{NM}: N\underset {A}{\otimes} M\longrightarrow B$ satisfying
the following commutative diagrams:
$$
\xymatrix{ M \underset {B}{\otimes} N \underset{A}{\otimes} M
\ar[rr]^{\hspace{8pt}\Phi_{MN} \otimes I_M} \ar[dd]^{I_M \otimes
\Psi_{NM}} && A
\underset{A}{\otimes} M \ar[dd]^{\cong} \\  &&\\
M \underset{B}{\otimes} B \ar[rr]^{\hspace{10pt}\cong} && M }
\hspace{6pt}{\rm and}\hspace{6pt} \xymatrix{ N \underset
{A}{\otimes} M \underset{B}{\otimes} N
\ar[rr]^{\hspace{8pt}\Psi_{NM}\otimes I_N} \ar[dd]^{I_N\otimes
\Phi_{MN}} && B
\underset{B}{\otimes} N \ar[dd]^{\cong}\\  &&\\
N \underset{A}{\otimes} A \ar[rr]^{\hspace{10pt}\cong} && N
\hspace{2pt}.}
$$
Let us write this Morita context as $(A, B, M, N, \Phi_{MN},
\Psi_{NM})$. We refer the reader to \cite{Morita} for the basic
properties of Morita contexts. If $(A, B, M, N,$ $ \Phi_{MN},
\Psi_{NM})$ is a Morita context, then the set
$$
\left[
\begin{array}
[c]{cc}%
A & M\\
N & B\\
\end{array}
\right]=\left\{ \left[
\begin{array}
[c]{cc}%
a& m\\
n & b\\
\end{array}
\right] \vline a\in A, m\in M, n\in N, b\in B \right\}
$$
form an $\mathcal{R}$-algebra under matrix-like addition and
matrix-like multiplication, where at least one of the two bimodules
$M$ and $N$ is distinct from zero. Such an $\mathcal{R}$-algebra is
usually called a \textit{generalized matrix algebra} of order $2$
and is denoted by
$$
\mathcal{G}=\mathcal{G}(A, M, N, B)=\left[
\begin{array}
[c]{cc}%
A & M\\
N & B\\
\end{array}
\right].
$$
In a similar way, one can define a generalized matrix algebra of
order $n>2$. It was shown that up to isomorphism, arbitrary
generalized matrix algebra of order $n$ $(n\geq 2)$ is a generalized
matrix algebra of order 2 \cite[Example 2.2]{LiWei1}. If one of the
modules $M$ and $N$ is zero, then $\mathcal{G}$ exactly degenerates
to an \textit{upper triangular algebra} or a \textit{lower
triangular algebra}. In this case, we denote the resulted upper
triangular algebra (resp. lower triangular algebra) by
$$\mathcal{T^U}=\mathcal{T}(A, M, B)=
\left[
\begin{array}
[c]{cc}%
A & M\\
O & B\\
\end{array}
\right]   \hspace{8pt} \left({\rm resp.} \hspace{4pt} \mathcal{T_L}=\mathcal{T}(A, N, B)=
\left[
\begin{array}
[c]{cc}%
A & O\\
N & B\\
\end{array}
\right]\right)
$$
Note that our current generalized matrix algebras contain those
generalized matrix algebras in the sense of Brown \cite{Brown} as
special cases. Let $\mathcal{M}_n(\mathcal{R})$ be the full matrix
algebra consisting of all $n\times n$ matrices over $\mathcal{R}$.
It is worth to point out that the notion of generalized matrix
algebras efficiently unifies triangular algebras with full matrix
algebras together. The feature of our systematic work
is to deal with all questions related to (non-)linear mappings of
triangular algebras and of full matrix algebras under a unified
frame, which is the admired generalized matrix algebras frame, see
\cite{FosnerWeiXiao, LiWei1, LiWei2, LiWykWei, XiaoWei1, XiaoWei2}.

Let us list some classical examples of generalized matrix algebras
which will be revisited in the sequel (Section \ref{xxsec3} and
Section \ref{xxsec4}). Since these examples have already been
presented in many papers, we just state their title without any
details.
\begin{enumerate}
\item[(\rm a)] Unital algebras with nontrivial
idempotents;
\item[(\rm b)] Full matrix algebras;
\item[(\rm c)] Inflated algebras;
\item[(\rm d)] Upper and lower triangular matrix algebras;
\item[(\rm e)] Nest algebras over a Hilbert space;
\item[(\rm f)] Standard operator algebras over a Banach space.
\end{enumerate}

\noindent{These} generalized matrix algebras regularly appear in the theory of associative algebras and noncommutative Noetherian algebras in the most diverse situations, which is due to its powerful persuasiveness and intuitive illustration effect. However, people pay less attention to the linear mappings of generalized matrix algebras. It was Krylov who initiated the study of linear mappings on generalized matrix algebras from the point of classifying view \cite{Krylov1}. Since then many articles are devoted to this topic, and a number of interesting results are obtained (see \cite{Benkovic5, BenkovicSirovnik, DuWang2, DuWang3, LiWei1, LiWei2, LiWykWei, QiHou3, XiaoWei1, XiaoWei2}). Nevertheless, it leaves so much to be desired. The representation theory, homological behavior, $K$-theory of generalized matrix algebras are intensively intesgivated by Krylov and his coauthors in \cite{Krylov2, Krylov3, KrylovTuganbaev}. We will propose some open questions concerning linear mappings and functional identities of generalized matrix algebras in Section \ref{xxsec5} of this article. Therefore, generalized matrix algebras are indeed one class of great potential and inspiring associative algebras. We can never emphasize on the importance of generalized matrix algebras too much.

\section{Centralizing traces of generalized matrix algebras.}\label{xxsec3}

In this section we will establish sufficient conditions for each
centralizing trace of an arbitrary bilinear mapping on a generalized
matrix algebra $\mathcal{G}$ to be proper (Theorem
\ref{xxsec3.4}). Consequently, we are able to describe centralizing traces of bilinear mappings on triangular algebras, on full matrix
algebras and on certain unital algebras with nontrivial idempotents.
The most important advantage is that Theorem \ref{xxsec3.4} will be used to
characterize Lie triple isomorphisms from a generalized matrix algebra into
another in Section \ref{xxsec4}.

Throughout this section, we denote the generalized matrix algebra of
order $2$ originated from the Morita context $(A, B, _AM_B, _BN_A,
\Phi_{MN}, \Psi_{NM})$ by
$$
\mathcal{G}=\mathcal{G}(A, M, N, B)=\left[
\begin{array}
[c]{cc}%
A & M\\
N & B
\end{array}
\right] ,
$$
where at least one of the two bimodules $M$ and $N$ is distinct from
zero. We always assume that $M$ is faithful as a left $A$-module and
also as a right $B$-module, but no any constraint conditions on $N$.
The center of $\mathcal{G}$ is
$$
\mathcal{Z(G)}=\left\{\hspace{2pt} a\oplus b \hspace{4pt} \vline \hspace{4pt} am=mb, \hspace{4pt} na=bn,\ \forall\
m\in M, \hspace{4pt} \forall n\in N \right\}.
$$
Indeed, by \cite[Lemma 1]{Krylov1} we know that the center
$\mathcal{Z(G)}$ consists of all diagonal matrices $a\oplus b$,
where $a\in \mathcal{Z}(A)$, $b\in \mathcal{Z}(B)$ and $am=mb$,
$na=bn$ for all $m\in M, n\in N$. However, in our situation which
$M$ is faithful as a left $A$-module and also as a right $B$-module,
the conditions that $a\in \mathcal{Z}(A)$ and $b\in \mathcal{Z}(B)$
become redundant and can be deleted. Indeed, if $am=mb$ for all
$m\in M$, then for any $a^\prime \in A$ we get
$$
(aa^\prime-a^\prime a)m=a(a'm)-a'(am)=(a'm)b-a'(mb)=0.
$$
The assumption that $M$ is faithful as a left $A$-module leads to
$aa'-a'a=0$ and hence $a\in \mathcal{Z}(A)$. Likewise, we also have
$b\in \mathcal{Z}(B)$.

Let us define two natural $\mathcal{R}$-linear projections
$\pi_A:\mathcal{G}\rightarrow A$ and $\pi_B:\mathcal{G}\rightarrow
B$ by
$$
\pi_A: \left[
\begin{array}
[c]{cc}%
a & m\\
n & b\\
\end{array}
\right] \longmapsto a \quad \text{and} \quad \pi_B: \left[
\begin{array}
[c]{cc}%
a & m\\
n & b\\
\end{array}
\right] \longmapsto b.
$$
By the above paragraph, it is not difficult to see that $\pi_A
\left(\mathcal{Z(G)}\right)$ is a subalgebra of $\mathcal{Z}(A)$ and
that $\pi_B\left(\mathcal{Z(G)}\right)$ is a subalgebra of
$\mathcal{Z}(B)$.
Given an element $a\in\pi_A(\mathcal{Z(G)})$, if $a\oplus b, a\oplus b^\prime \in \mathcal{Z(G)}$, then we have $am=mb=mb'$ for
all $m\in M$. Since $M$ is faithful as a right $B$-module,
$b=b^\prime$. That implies there exists a unique
$b\in\pi_B(\mathcal{Z(G)})$, which is denoted by $\varphi(a)$, such
that $
a\oplus b \in \mathcal{Z(G)}$. It is easy to
verify that the map $\varphi:\pi_A(\mathcal{Z(G)})\longrightarrow
\pi_B(\mathcal{Z(G)})$ is an algebraic isomorphism such that
$am=m\varphi(a)$ and $na=\varphi(a)n$ for all $a\in
\pi_A(\mathcal{Z(G)}), m\in M, n\in N$.

Let $\mathcal{A}$ and $\mathcal{B}$ be algebras. Recall an
$(\mathcal{A}, \mathcal{B})$-bimodule $\mathcal{M}$ is \textit{loyal} if
$a\mathcal{M}b=0$ implies that $a=0$ or $b=0$ for all $a\in
\mathcal{A}, b\in \mathcal{B}$. Obviously, each loyal $(\mathcal{A}, \mathcal{B})$-bimodule $\mathcal{M}$ is faithful as a left $\mathcal{A}$-module and also as a right $\mathcal{B}$-module. We now list some basic facts related to generalized matrix algebras, which can be found in \cite{XiaoWei2}.

\begin{lemma}{\rm (\cite[Lemma 2.3]{BenkovicEremita1})}\label{xxsec3.1}
Let $M$ be a loyal $(A, B)$-bimodule and let $f, g$ be two arbitrary mappings from $M$ into $A$. Suppose $f(m)n+g(n)m=0$ for all ${m,n\in M}$. If B is noncommutative, then $f=g=0$.
\end{lemma}

\begin{lemma}\label{xxsec3.2}
Let $\mathcal{G}=\mathcal{G}(A, M, N, B)$ be a generalized matrix algebra with a loyal $(A,B)$-bimodule $M$.
Let ${\alpha\in \pi_A(\mathcal{Z(\mathcal{G})})}$ and let $0\neq a\in A$.
If $\alpha a=0$, then $\alpha=0$.
\end{lemma}

\begin{lemma}\label{xxsec3.3}
Let $\mathcal{G}=\mathcal{G}(A, M, N, B)$ be a generalized matrix algebra with a loyal $(A,B)$-bimodule $M$. 
Then the center $\mathcal{Z(G)}$ of $\mathcal{G}$ is a domain.
\end{lemma}

In this section we will establish sufficient conditions for each centralizing
trace of an arbitrary bilinear mapping on generalized matrix algebra
$\mathcal{G}=\mathcal{G}(A, M, N, B)$ to be proper.

\begin{theorem}\label{xxsec3.4}
Let $\mathcal{G}=\mathcal{G}(A, M, N, B)$  be a $2$-torsionfree generalized matrix algebra over a
commutative ring $\mathcal{R}$ and ${\mathfrak q}\colon
\mathcal{G}\times \mathcal{G}\longrightarrow \mathcal{G}$ be an
$\mathcal{R}$-bilinear mapping. If
\begin{enumerate}
\item[{\rm(1)}] every commuting linear mapping on $A$ or $B$ is proper;
\item[{\rm(2)}] $\pi_A(\mathcal{Z(G)})=\mathcal{Z}(A)\neq A$ and
$\pi_B(\mathcal{Z(G)})=\mathcal{Z}(B)\neq
B$;
\item[{\rm(3)}] $M$ is loyal,
\end{enumerate}
then every centralizing trace ${\mathfrak T}_{\mathfrak q}:
\mathcal{G}\longrightarrow \mathcal{G}$ of ${\mathfrak q}$ is
proper.
\end{theorem}

For convenience, let us write $A_1=A$, $A_2=M$, $A_3=N$ and $A_4=B$.
We denote the identity of $A_{1}$ by $1$ and the identity of $A_{4}$ by
$1^\prime$. Suppose that ${\mathfrak T}_{\mathfrak q}$ is the trace of the $\mathcal{R}$-bilinear mapping ${\mathfrak q}$. Then there exist
$\mathcal{R}$-bilinear mappings $f_{ij}: A_i\times A_j\rightarrow
A_1$, $g_{ij}: A_i\times A_j\rightarrow A_2$, $h_{ij}: A_i\times
A_j\rightarrow A_3$ and $k_{ij}: A_i\times A_j\rightarrow A_4$,
$1\leqslant i\leqslant j\leqslant 4$, such that
$$\begin{aligned}
{\mathfrak T}_{\mathfrak q}: \mathcal{G}& \longrightarrow \mathcal{G}\\
\left[
\begin{array}
[c]{cc}%
a_1 & a_2\\
a_3 & a_4\\
\end{array}
\right]& \longmapsto \left[
\begin{array}
[c]{cc}%
F(a_1,a_2,a_3,a_4) & G(a_1,a_2,a_3,a_4)\\
H(a_1,a_2,a_3,a_4) & K(a_1,a_2,a_3,a_4)\\
\end{array}
\right], \hspace{3pt} \forall \left[\begin{array}
[c]{cc}%
a_1 & a_2\\
a_3 & a_4\\
\end{array}\right]\in \mathcal{G},
\end{aligned}
$$
where
$$
F(a_1,a_2,a_3,a_4)=\sum_{1\leqslant i\leqslant j\leqslant
4}f_{ij}(a_i,a_j),
$$
$$ G(a_1,a_2,a_3,a_4)=\sum_{1\leqslant i\leqslant
j\leqslant 4}g_{ij}(a_i,a_j),
$$
$$
H(a_1,a_2,a_3,a_4)=\sum_{1\leqslant i\leqslant j\leqslant
4}h_{ij}(a_i,a_j),
$$
$$ K(a_1,a_2,a_3,a_4)=\sum_{1\leqslant i\leqslant
j\leqslant 4}k_{ij}(a_i,a_j).
$$
Since ${\mathfrak T}_{\mathfrak q}$ is centralizing, we have
$$\begin{aligned}
 & \left[\left[
\begin{array}
[c]{cc}%
F & G\\
H & K\\
\end{array}
\right],\left[
\begin{array}
[c]{cc}%
a_1 & a_2\\
a_3 & a_4\\
\end{array}
\right]\right]\\
=&\left[
\begin{array}
[c]{cc}%
Fa_1+Ga_3-a_1F-a_2H & Fa_2+Ga_4-a_1G-a_2K\\
Ha_1+Ka_3-a_3F-a_4H & Ha_2+Ka_4-a_3G-a_4K\\
\end{array}
\right]\in \mathcal{Z(G)}
\end{aligned} \eqno(3.1)
$$
for all $\left[
\begin{array}
[c]{cc}%
a_1 & a_2\\
a_3 & a_4\\
\end{array}
\right]\in \mathcal{G}$.

We notice that the proof of Theorem \ref{xxsec3.4} is similar with
that of \cite[Theorem 3.4]{FosnerWeiXiao}, but with much more complicated computational process.
Now we divide the proof of Theorem \ref{xxsec3.4}
to a series of lemmas for comfortable reading.

\begin{lemma}\label{xxsec3.5}
If $K\colon A_{4}\times A_{4}\longrightarrow A_{3}$ {\rm (}resp. $K\colon A_{1}\times A_{1}\longrightarrow A_{2}${\rm )}
is an $R$-bilinear mapping with $xK(x,x)=0$ {\rm (}resp. $xK(x,x)=0${\rm )} for all $x\in A_4$
{\rm (}resp. for all $x\in A_1${\rm )}, then $K(x,x)=0$. 
\end{lemma}

\begin{proof}
Setting $x=1'$, we can obtain that $K(1',1')=0$.
Replacing $x$ by $x+1'$ in $K(x,x)x=0$ gives
$$
K(x,x)=-(1'+x)(K(1',x)+K(x,1')). \eqno(3.2)
$$
Substituting $x-1^\prime$ for $x$ in $K(x,x)x=0$,
we have
$$
K(x,x)=(1'-x)(K(1',x)+K(x,1')).  \eqno(3.3)
$$
Comparing the above two relations we get $(K(1',x)+K(x,1'))=0$. And hence $K(x,x)=0$.
\end{proof}

\begin{lemma}\label{xxsec3.6}{\rm(\cite[Lemma 3.5]{XiaoWei2})}
For any $a_1\in A_1, a_2\in A_2, a_3\in A_3\in A_3, a_4\in A_4$, we have 
$$
H(a_{1},a_{2},a_{3},a_{4})=h_{13}(a_{1},a_{3})+ h_{23}(a_{2},a_{3})
+ h_{33}(a_{3},a_{3})+ h_{34}(a_{3},a_{4}). \eqno(3.4)
$$
\end{lemma}

\begin{lemma}\label{xxsec3.7} {\rm(\cite[Lemma 3.6]{XiaoWei2})}
For any $a_1\in A_1, a_2\in A_2, a_3\in A_3\in A_3, a_4\in A_4$, we have
$$
G(a_{1},a_{2},a_{3},a_{4})=g_{12}(a_{1},a_{2})+ g_{22}(a_{2},a_{2})
+ g_{23}(a_{2},a_{3})+ g_{24}(a_{2},a_{4}). \eqno(3.5)
$$
\end{lemma}

\begin{lemma}\label{xxsec3.8}
With the notations as above, we have

\begin{enumerate}

\item[(1)]
$a_{1}\longmapsto f_{11}(a_{1},a_{1})$ is a commuting trace,\\
$a_{1}\longmapsto f_{12}(a_{1},a_{2})$ is a commuting linear mapping for each $a_{2}\in A_{2}$,\\
$a_{1}\longmapsto f_{13}(a_{1},a_{3})$ is a commuting linear mapping for each $a_{3}\in A_{3}$,\\
$a_{4}\longmapsto k_{24}(a_{2},a_{4})$ is a commuting linear mapping for each $a_{2}\in A_{2}$,\\
$a_{4}\longmapsto k_{34}(a_{3},a_{4})$ is a commuting linear mapping for each $a_{3}\in A_{3}$,\\
$a_{4}\longmapsto k_{44}(a_{4},a_{4})$ is a commuting trace.

\item[(2)]
$[f_{14}(a_{1},a_{4}),a_{1}]=\varphi^{-1}([k_{11}(a_{1},a_{1}),a_{4}])\in \mathcal{Z}(A_{1})$, \\
$[f_{24}(a_{2},a_{4}),a_{1}]=\varphi^{-1}([k_{12}(a_{1},a_{2}),a_{4}])\in \mathcal{Z}(A_{1})$, \\
$[f_{34}(a_{3},a_{4}),a_{1}]=\varphi^{-1}([k_{13}(a_{1},a_{3}),a_{4}])\in \mathcal{Z}(A_{1})$, \\
$[f_{44}(a_{4},a_{4}),a_{1}]=\varphi^{-1}([k_{14}(a_{1},a_{4}),a_{4}])\in \mathcal{Z}(A_{1})$.

\item[(3)]
$f_{22}(a_{2},a_{2})\in \mathcal{Z}(A_{1})$ and $f_{33}(a_{3},a_{3})\in \mathcal{Z}(A_{1})$, \\
$k_{22}(a_{2},a_{2})\in \mathcal{Z}(A_{4})$ and $k_{33}(a_{3},a_{3})\in \mathcal{Z}(A_{4})$.

\item[(4)]
$g_{24}(a_{2},a_{4})a_{3}-a_{2}h_{34}(a_{3},a_{4})=\varphi^{-1}([k_{23}(a_{2},a_{3}),a_{4}]
-a_{3}g_{24}(a_{2},a_{4})+h_{34}(a_{3},a_{4})a_{2})$,\\
$h_{13}(a_{1},a_{3})a_{2}-a_{3}g_{12}(a_{1},a_{2})=\varphi([f_{23}(a_{2},a_{3}),a_{1}]
+g_{12}(a_{1},a_{2})a_{3}-a_{2}h_{13}(a_{1},a_{3}))$.
\end{enumerate}
\end{lemma}

\begin{proof}
It follows from the equation $(3.1)$ that 
$$
[F,a_{1}]+Ga_{3}-a_{2}H=\varphi^{-1}([K,a_{4}]+Ha_{2}-a_{3}G)  \eqno(3.6)
$$
for all $a_{i}\in A_{i} (1\leq i\leq 4)$.
Taking $a_{2}=0, a_{3}=0$ in $(3.6)$ we obatin
$$
\begin{aligned}
&[f_{11}(a_{1},a_{1})+f_{14}(a_{1},a_{4})+f_{44}(a_{4},a_{4}),a_{1}] \\
&=\varphi^{-1}([k_{11}(a_{1},a_{1})+k_{14}(a_{1},a_{4})+k_{44}(a_{4},a_{4}),a_{4}])
\end{aligned} \eqno(3.7)
$$
for all $a_{1}\in A_{1}, a_{4}\in A_{4}$. Let us choose $a_{4}=0$. Thus
$[f_{11}(a_{1},a_{1}),a_{1}]=0$. Using similar computation, we get $[k_{44}(a_{4},a_{4}),a_{4}]=0$.
Replacing $a_{4}$ by $-a_{4}$ in $(3.7)$ yields that
$$
[f_{14}(a_{1},a_{4}),a_{1}]=\varphi^{-1}([k_{11}(a_{1},a_{1}),a_{4}]) \eqno(3.8)
$$
and
$$
[f_{44}(a_{4},a_{4}),a_{1}]=\varphi^{-1}([k_{14}(a_{1},a_{4}),a_{4}]) \eqno(3.9)
$$
for all $a_{1}\in A_{1}, a_{4}\in A_{4}$. Let us take $a_{1}=0, a_{4}=0$ in $(3.6)$. Then we arrive at that
$$
\begin{aligned}
&(g_{22}(a_2,a_2)+g_{23}(a_2,a_3))a_3-a_2(h_{23}(a_2,a_3)+h_{33}(a_3,a_3)) \\
&=\varphi^{-1}((h_{23}(a_2,a_3)+h_{33}(a_3,a_3))a_2-a_3(g_{22}(a_2,a_2)+g_{23}(a_2,a_3)))
\end{aligned}\eqno(3.10)
$$
for all $a_{1}\in A_{1}, a_{2}\in A_{2}, a_{3}\in A_{3}$. Replacing $a_{2}$ by $-a_{2}$ in $(3.10)$ we have
$$
g_{22}(a_2,a_2)a_3-a_2h_{23}(a_2,a_3)=\varphi^{-1}(h_{23}(a_2,a_3)a_2-a_3g_{22}(a_2,a_2)) \eqno(3.11)
$$
and
$$
g_{23}(a_2,a_3)a_3-a_2h_{33}(a_3,a_3)=\varphi^{-1}(h_{33}(a_3,a_3)a_2-a_3g_{23}(a_2,a_3)) \eqno(3.12)
$$
for all $a_{2}\in A_{2}, a_{3}\in A_{3}, a_{4}\in A_{4}$.

Combining $(3.6)$ with $(3.11)$ and $(3.12)$ we conclude that
$$
\begin{aligned}
&[f_{12}(a_{1},a_{2})+f_{13}(a_{1},a_{3})+f_{22}(a_{2},a_{2})+f_{23}(a_{2},a_{3}) \\
&+f_{24}(a_{2},a_{4})+f_{33}(a_{3},a_{3})+f_{34}(a_{3},a_{4}),a_{1}] \\
&+(g_{12}(a_1,a_2)+g_{24}(a_2,a_4))a_{3}-a_{2}(h_{13}(a_1,a_3)+h_{34}(a_3,a_4)) \\
=&\varphi^{-1}([k_{12}(a_{1},a_{2})+k_{13}(a_{1},a_{3})+k_{22}(a_{2},a_{2}) \\
&+k_{23}(a_{2},a_{3})+k_{24}(a_{2},a_{4})+k_{33}(a_{3},a_{3})+k_{34}(a_{3},a_{4}),a_{4}] \\
&-a_{3}(g_{12}(a_1,a_2)+g_{24}(a_2,a_4))+(h_{13}(a_1,a_3)+h_{34}(a_3,a_4))a_2)
\end{aligned}\eqno(3.13)
$$
for all $a_{1}\in A_{1}, a_{2}\in A_{2}, a_{3}\in A_{3}, a_{4}\in A_{4}$.
Substituting  $-a_{1}$ for $a_{1}$ in $(3.13)$ we get 
$$
\begin{aligned}
&[f_{12}(a_{1},a_{2})+f_{13}(a_{1},a_{3}),a_1]+g_{24}(a_2,a_4)a_{3}-a_{2}h_{34}(a_3,a_4) \\
=&\varphi^{-1}([k_{22}(a_{2},a_{2})+k_{23}(a_{2},a_{3})+k_{24}(a_{2},a_{4})
+k_{33}(a_{3},a_{3})+k_{34}(a_{3},a_{4}),a_{4}] \\
&-a_{3}g_{24}(a_2,a_4)+h_{34}(a_3,a_4)a_2)
\end{aligned}\eqno(3.14)
$$
and
$$
\begin{aligned}
&[f_{22}(a_{2},a_{2})+f_{23}(a_{2},a_{3})+f_{24}(a_{2},a_{4})
+f_{33}(a_{3},a_{3})+f_{34}(a_{3},a_{4}),a_{1}] \\
&+g_{12}(a_1,a_2)a_{3}-a_{2}h_{13}(a_1,a_3) \\
=&\varphi^{-1}([k_{12}(a_{1},a_{2})+k_{13}(a_{1},a_{3}),a_{4}]
-a_{3}g_{12}(a_1,a_2)+h_{13}(a_1,a_3)a_2)
\end{aligned}\eqno(3.15)
$$
for all $a_{1}\in A_{1}, a_{2}\in A_{2}, a_{3}\in A_{3}, a_{4}\in A_{4}$.
Let us take $a_{3}=0, a_{4}=0$; $a_{1}=0, a_{3}=0$; $a_{2}=0, a_{4}=0$;
$a_{1}=0, a_{2}=0$ in $(3.14)$ and $(3.15)$, respectively. Thus we obtain 
$$
\begin{aligned}
&[f_{12}(a_{1},a_{2}),a_1]=0,\quad [f_{22}(a_{2},a_{2}),a_1]=0, \\
&[k_{22}(a_{2},a_{2}),a_4]=0,\quad [k_{24}(a_{2},a_{4}),a_4]=0, \\
&[f_{13}(a_{1},a_{2}),a_3]=0,\quad [f_{33}(a_{3},a_{3}),a_1]=0, \\
&[k_{34}(a_{3},a_{4}),a_4]=0,\quad [k_{33}(a_{3},a_{3}),a_4]=0
\end{aligned}
$$
for all $a_{1}\in A_{1}, a_{2}\in A_{2}, a_{3}\in A_{3}, a_{4}\in A_{4}$.
Let us choose $a_{2}=0, a_{3}=0$ in $(3.15)$, respectively. We have
$$
[f_{34}(a_{3},a_{4}),a_{1}]=\varphi^{-1}([k_{13}(a_{1},a_{3}),a_{4}]) \eqno(3.16)
$$
and
$$
[f_{24}(a_{2},a_{4}),a_{1}]=\varphi^{-1}([k_{12}(a_{1},a_{2}),a_{4}]) \eqno(3.17)
$$
for all $a_{1}\in A_{1}, a_{2}\in A_{2}, a_{3}\in A_{3}, a_{4}\in A_{4}$.
Setting $a_{1}=0$ in $(3.14)$ gives
$$
\begin{aligned}
&g_{24}(a_{2},a_{4})a_3-a_2h_{34}(a_{3},a_{4}) \\
=&\varphi^{-1}([k_{23}(a_{2},a_{3}),a_4]-a_3g_{24}(a_{2},a_{4})+h_{34}(a_{3},a_{4})a_2)
\end{aligned} \eqno(3.18)
$$
for all $a_{2}\in A_{2}, a_{3}\in A_{3}, a_{4}\in A_{4}$. Putting $a_{4}=0$ in $(3.15)$ yields
$$
\begin{aligned}
&[f_{23}(a_{2},a_{3}),a_1]+g_{12}(a_{1},a_2)a_3-a_2h_{13}(a_1,a_3) \\
=&\varphi^{-1}(h_{13}(a_1,a_3)a_2-a_3g_{12}(a_1,a_2))
\end{aligned} \eqno(3.19)
$$
for all $a_{1}\in A_{1}, a_{2}\in A_{2}, a_{3}\in A_{3}$.
\end{proof}

\begin{lemma}{\rm (}\cite[Lemma 3.9]{XiaoWei2}{\rm )}\label{xxsec3.9}
We have $f_{22}(a_2,a_2) \oplus k_{22}(a_2,a_2)\in \mathcal{Z(G)}$ and
$f_{33}(a_3,a_3) \oplus k_{33}(a_3,a_3)\in \mathcal{Z(G)}$ for all 
$a_2\in A_2, a_3\in A_3$.
\end{lemma}

\begin{lemma}\label{xxsec3.10}
There exist linear mapping $\xi\colon A_2\longrightarrow {\mathcal Z}(A_{1})$ and bilinear mapping $\eta\colon A_1\times A_2\longrightarrow {\mathcal Z}(A_{1})$ such that
$f_{12}(a_{1},a_{2})=\xi(a_{2})a_{1}+\eta(a_{1},a_{2})$ for all $a_1\in A_1, a_2\in A_2$.
\end{lemma}

\begin{proof}
Since $a_1\longmapsto f_{12}(a_{1},a_{2})$ is commuting linear mapping for each $a_2\in A_2$, there exist mappings $\xi\colon A_2\longrightarrow {\mathcal Z}(A_{1})$
and $\eta\colon A_1\times A_2\longrightarrow {\mathcal Z}(A_{1})$ such that
$$
f_{12}(a_{1},a_{2})=\xi(a_{2})a_{1}+\eta(a_{1},a_{2}),
$$
where $\eta$ is $R$-linear in the first argument for all $a_1\in A_1, a_2\in A_2$. Let us show that $\xi$ is $R$-linear and $\eta$ is $R$-bilinear.
It is easy to check that
$$
\begin{aligned}
f_{12}(a_{1},a_{2}+b_{2})&=\xi(a_2+b_2)a_1+\eta(a_1,a_2+b_2) \\
f_{12}(a_{1},a_{2})+f_{12}(a_{1},b_{2})&=\xi(a_2)a_1+\eta(a_1,a_2)+\xi(b_2)a_1+\eta(a_1,b_2)
\end{aligned}
$$
and
$$
\begin{aligned}
&(\xi(a_2+b_2)-\xi(a_2)-\xi(b_2))a_1 \\
&+\eta(a_1,a_2+b_2)-\eta(a_1,a_1)-\eta(a_1,b_2)=0
\end{aligned}
$$
for all $a_1\in  A_1, a_2, b_2\in  A_2$. Note that $\xi$ and $\eta$ map into ${\mathcal Z}(A_1)$. Thus $(\xi(a_2+b_2)-\xi(a_2)-\xi(b_2))[a_1,b_1]=0$ for all $a_1, b_1\in  A_{1}$, and $a_2, b_2\in  A_2$.
Since $A_1$ is non-commutative, $\xi(a_2+b_2)=\xi(a_2)+\xi(b_2)$.
That is, $\xi\colon A_2\longrightarrow {\mathcal Z}(A_1)$ is an $\mathcal{R}$-linear mapping.
Consequently, $\eta$ is $\mathcal{R}$-linear in the second argument and $\eta\colon A_1\times A_2\longrightarrow {\mathcal Z}(A_1)$ is a bilinear mapping.
\end{proof}

\begin{lemma}\label{xxsec3.11}
$f_{44}(a_4,a_4)\in {\mathcal Z}(A_1)$ and $k_{11}(a_1,a_1)\in {\mathcal Z}(A_4)$.
\end{lemma}

\begin{proof}
It follows from the equation $(3.1)$ that
$$
Ha_{1}+Ka_{3}-a_{3}F-a_{4}H=0. \eqno(3.20)
$$
Let us take $a_{4}=0$ in $(3.20)$. Using Lemma \ref{xxsec3.9} we get
$$
\begin{aligned}
&(h_{13}(a_{1},a_{3})+h_{23}(a_{2},a_{3})+h_{33}(a_{3},a_{3}))a_1\\
+&(k_{11}(a_{1},a_{1})+k_{12}(a_{1},a_{1})+k_{13}(a_{1},a_{3})+k_{23}(a_{2},a_{3}))a_{3} \\
-&a_{3}(f_{11}(a_{1},a_{1})+f_{12}(a_{1},a_{1})+f_{13}(a_{1},a_{3})+f_{23}(a_{2},a_{3}))=0.
\end{aligned} \eqno(3.21)
$$
Replacing $a_1$ by $-a_1$ in $(3.21)$ we get
$$
\begin{aligned}
&h_{13}(a_{1},a_{3})a_1+(k_{11}(a_{1},a_{1})+k_{23}(a_{2},a_{3}))a_{3} \\
-&a_{3}(f_{11}(a_{1},a_{1})+f_{23}(a_{2},a_{3}))=0
\end{aligned} \eqno(3.22)
$$
and
$$
\begin{aligned}
&(h_{23}(a_{2},a_{3})+h_{33}(a_{3},a_{3}))a_1+(k_{12}(a_{1},a_{2})+k_{13}(a_{1},a_{3}))a_{3} \\
-&a_{3}(f_{12}(a_{1},a_{2})+f_{13}(a_{1},a_{3}))=0.
\end{aligned}  \eqno(3.23)
$$
Combining $(3.22)$ with $(3.23)$ yields
$$
h_{23}(a_{2},a_{3})a_1=a_3f_{12}(a_{1},a_{2})-k_{12}(a_{1},a_{2})a_3, \eqno(3.24)
$$
$$
h_{33}(a_{3},a_{3})a_1=a_3f_{13}(a_{1},a_{3})-k_{13}(a_{1},a_{3})a_3, \eqno(3.25)
$$
$$
h_{13}(a_{1},a_{3})a_1=a_3f_{11}(a_{1},a_{1})-k_{11}(a_{1},a_{1})a_3, \eqno(3.26)
$$
$$
a_3f_{23}(a_{2},a_{3})=k_{23}(a_{2},a_{3})a_3. \eqno(3.27)
$$
Taking $a_1=0$ in $(3.20)$ and using similar computational procedures, we obtain
$$
a_4h_{34}(a_{3},a_{4})=k_{44}(a_{4},a_{4})a_3-a_3f_{44}(a_{4},a_{4}), \eqno(3.28)
$$
$$
a_4h_{33}(a_{3},a_{3})=k_{34}(a_{3},a_{4})a_3-a_3f_{34}(a_{3},a_{4}),  \eqno(3.29)
$$
$$
a_4h_{23}(a_{2},a_{3})=k_{24}(a_{2},a_{4})a_3-a_3f_{24}(a_{2},a_{4}),  \eqno(3.30)
$$

In view of the relation $(3.1)$ we have
$$
Fa_{2}+Ga_{4}-a_{1}G-a_{2}K=0. \eqno(3.31)
$$
Considering $(3.31)$ and using similar computational procedures yield
$$
a_1g_{22}(a_{2},a_{2})=f_{12}(a_{1},a_{2})a_2-a_2k_{12}(a_{1},a_{2}), \eqno(3.32)
$$
$$
a_1g_{23}(a_{2},a_{3})=f_{13}(a_{1},a_{3})a_2-a_2k_{13}(a_{1},a_{3}), \eqno(3.33)
$$
$$
a_1g_{12}(a_{1},a_{2})=f_{11}(a_{1},a_{1})a_2-a_2k_{11}(a_{1},a_{1}), \eqno(3.34)
$$
$$
f_{23}(a_{2},a_{3})a_2=a_2k_{23}(a_{2},a_{3}), \eqno(3.35)
$$
$$
g_{24}(a_{2},a_{4})a_4=a_2k_{44}(a_{4},a_{4})-f_{44}(a_{4},a_{4})a_2, \eqno(3.36)
$$
$$
g_{23}(a_{2},a_{3})a_4=a_2k_{34}(a_{3},a_{4})-f_{34}(a_{3},a_{4})a_2, \eqno(3.37)
$$
$$
g_{22}(a_{2},a_{2})a_4=a_2k_{24}(a_{2},a_{4})-f_{24}(a_{2},a_{4})a_2. \eqno(3.38)
$$

The identities $(3.24)-(3.30)$ together with $(3.20)$ imply that
$$
\begin{aligned}
h_{34}(a_{3},a_{4})a_1+k_{14}(a_{1},a_{4})a_3
-a_3f_{14}(a_{1},a_{4})-a_4h_{13}(a_{1},a_{3})=0.
\end{aligned}  \eqno(3.39)
$$
Replacing $a_{3}$ by $a_{4}a_{3}$ in $(3.28)$ and substracting the left  multiplication of $(3.28)$ by $a_{4}$ give $a_{4}(h_{34}(a_{4}a_{3},a_{4})-a_{4}h_{34}(a_{3},a_{4}))=0$.
Let us take $K(x,y)=h_{34}(xa_{3},y)-xh_{34}(a_{3},y)$,
where $x,y\in A_{4}$. We can know that $K(x,y)\colon A_{4}\times A_{4}\longrightarrow A_{3}$ is an $\mathcal{R}$-bilinear mapping
and that $a_{4}K(a_{4},a_{4})=0$.
According to Lemma \ref{xxsec3.5}, we get
$$
h_{34}(a_{4}a_{3},a_{4})=a_{4}h_{34}(a_{3},a_{4}) \eqno(3.40)
$$
for all $a_{3}\in A_{3}, a_{4}\in A_{4}$.

Using $(3.26)$ and those similar computational procedures, we also have
$$
h_{13}(a_{1},a_{3}a_{1})=h_{13}(a_{1},a_{3})a_{1}. \eqno(3.41)
$$
Replacing $a_{3}$ by $a_{4}a_{3}$ in $(3.26)$ and substracting the left multiplcation of  $(3.26)$ by  $a_4$ gives 
$$
(h_{13}(a_{1},a_{4}a_{3})-a_{4}h_{13}(a_{1},a_{3}))a_{1}=-[k_{11}(a_{1},a_{1}),a_{4}]a_{3}\eqno(3.42)
$$
Substituting $a_{4}a_{3}$ for $a_{3}$ in $(3.39)$, substracting the left multiplication of $(3.39)$ by $a_{4}$ and using $(3.40)$ together we conclude
$$
[k_{14}(a_{1},a_{4}),a_{4}]a_{3}=a_{4}(h_{13}(a_{1},a_{4}a_{3})-a_{4}h_{13}(a_{1},a_{3})) \eqno(3.43)
$$
Combining the right multiplication of $(3.43)$ by $a_1$ with the left multiplication of $(3.42)$ by $a_4$ we arrive at
$$
[k_{14}(a_{1},a_{4}),a_{4}]a_{3}a_{1}=(-a_{4}[k_{11}(a_{1},a_{1}),a_{4}])a_3   \eqno(3.44)
$$
By $(3.9)$ and $(3.44)$ it follows that 
$$
a_{3}(-a_{1}[f_{44}(a_{4},a_{4}),a_{1}])=a_{4}[k_{11}(a_{1},a_{1}),a_{4}]a_{3}.  \eqno(3.45)
$$

In a similar way, the equations $(3.32)-(3.38)$ together with $(3.31)$ imply that
$$
g_{12}(a_1,a_2)a_4+f_{14}(a_{1},a_{4})a_2=a_1g_{24}(a_2,a_4)+a_2k_{14}(a_1,a_4) \eqno(3.46)
$$
Replacing $a_{2}$ by $a_{1}a_{2}$ in $(3.34)$ and then substracting the left multiplication of $(3.34)$ by $a_{1}$ we have
$$
g_{12}(a_1,a_1a_2)=a_1g_{12}(a_1,a_2)  \eqno(3.47)
$$
Substituting  $a_{2}a_{4}$ for $a_{2}$ in $(3.36)$ and substracting the right multiplication of $(3.36)$ by $a_{4}$, one get
$$
g_{24}(a_2a_4,a_4)=g_{24}(a_2,a_4)a_4  \eqno(3.48)
$$
Replacing $a_{2}$ by $a_{1}a_{2}$ in $(3.36)$ and substracting the left multiplication of $(3.36)$ by $a_{1}$, we obtain
$$
(g_{24}(a_1a_2,a_4)-a_1g_{24}(a_2,a_4))a_4=(-[f_{44}(a_4,a_4),a_1])a_2  \eqno(3.49)
$$
Substituting  $a_{1}a_{2}$ for $a_{2}$ in $(3.46)$, substracting the left multiplication of $(3.46)$ by $a_{1}$ and using $(3.47)$ together, we obtain
$$
[f_{14}(a_1,a_4),a_1]a_2=a_1(g_{24}(a_1a_2,a_4)-a_1g_{24}(a_2,a_4)). \eqno(3.50)
$$
Combining the left multiplication of $(3.49)$ by $a_1$ and right multiplication of $(3.50)$ by $a_4$, and using $(3.8)$ yields
$$
(-[f_{44}(a_{4},a_{4}),a_{1}])a_{1}a_{2}=a_{2}a_{4}[k_{11}(a_{1},a_{1}),a_{4}].  \eqno(3.51)
$$
Considering the center of generalized matrix algebras and combining $(3.45)$ with $(3.51)$ gives
$$
(-[f_{44}(a_{4},a_{4}),a_{1}]a_{1})\oplus a_{4}[k_{11}(a_{1},a_{1}),a_{4}]\in {\mathcal Z}(\mathcal{G}).
$$
Then $a_{1}[f_{44}(a_{4},a_{4}),a_{1}]\in {\mathcal Z}(A_{1})$ for all $a_{1}\in A_{1}$,
$a_{4}\in A_{4}$. Commuting with $b_{1}\in A_{1}$ gives $[f_{44}(a_{4},a_{4}),a_{1}][a_{1},b_{1}]=0$. If $a_1\notin \mathcal{Z}(A_1)$, then there exists $b_1\in A_1$ such that $[a_1,b_1]\neq 0$. In view of Lemma \ref{xxsec3.2}, we get $[f_{44}(a_4,a_4),a_1]=0$ for all $a_1\in A_1$. That is, $f_{44}(a_{4},a_{4})\in {\mathcal Z}(A_{1})$.
It follows from the relation $(3.9)$ that $[k_{14}(a_{1},a_{4}),a_{4}]=0$. This implies that $a_{4}\longmapsto k_{14}(a_{1},a_{4})$ is a commuting linear mapping for each $a_{1}\in A_{1}$. Similarly, one can show $[k_{11}(a_{1},a_{1}),a_{4}]=0$. That is, $k_{11}(a_{1},a_{1})\in {\mathcal Z}(A_{4})$. By the relation $(3.8)$ we obtain $[f_{14}(a_{1},a_{4}),a_{1}]=0$. 
That is, $a_{1}\longmapsto f_{14}(a_{1},a_{4})$ is a commuting linear mapping for each $a_{4}\in A_{4}$.
\end{proof}

\begin{lemma}\label{xxsec3.12}
$f_{34}(a_{3},a_4)\in \mathcal{Z}(A_1)$  and $k_{13}(a_1,a_3)\in \mathcal{Z}(A_4)$.
\end{lemma}

\begin{proof}
Let us set $a_4=1'$ in $(3.16)$. Then we get $f_{34}(a_3,1')\in \mathcal{Z}(A_{1})$. Taking $a_4=1'$ into $(3.37)$ leads to
$$
g_{23}(a_2,a_3)=a_2\omega(a_3), \eqno(3.52)
$$
where $\omega(a_3)=k_{34}(a_3, 1^\prime)-\varphi(f_{34}(a_3, 1^\prime))$ for all $a_2\in  A_2$, $a_3\in  A_3$.
It follows from $(3.37)$ and $(3.52)$ that
$$
a_2(\omega(a_{3})a_4-k_{34}(a_3,a_4))=-f_{34}(a_{3},a_4)a_{2}.\eqno(3.53)
$$
Replacing $a_2$ by $a_1a_2$ in $(3.53)$ and substracting the left multiplication of $(3.53)$ by $a_1$ yields $[f_{34}(a_{3},a_4),a_1]a_{2}=0$. Note that $M$ is loyal as an $(A, B)$-bimodule. Thus $[f_{34}(a_{3},a_4),a_1]=0$. According to $(3.16)$, we obtain $[k_{13}(a_1,a_3),a_4]=0$.
\end{proof}

\begin{lemma}\label{xxsec3.13}
$f_{24}(a_2,a_4)\in {\mathcal Z}(A_{1})$ and $k_{12}(a_{1},a_{2})\in {\mathcal Z}(A_{4})$.
\end{lemma}

\begin{proof}
Let us set $a_{1}=1$ in $(3.17)$. Then we have $k_{12}(1,a_{2})\in {\mathcal Z}(A_{4})$. Taking $a_{1}=1$ into $(3.32)$ gives
$$
g_{22}(a_2,a_2)=\psi(a_2)a_2, \eqno(3.54)
$$
where $\psi(a_2)=(f_{12}(1,a_2)-\varphi^{-1}(k_{12}(1,a_2)))$ for all $a_2\in  A_2$.
Combining $(3.32)$ with Lemma \ref{xxsec3.10} yields
$$
a_1(\psi(a_2)-\xi(a_2))a_2=a_{2}(\varphi(\eta(a_1,a_2))-k_{12}(a_1,a_2)).\eqno(3.55)
$$

Let us write $Y(a_{2})=\psi(a_2)-\xi(a_2)$, $X(a_1,a_2)=\varphi(\eta(a_1,a_2))-k_{12}(a_{1},a_{2})$.
Setting $a_{1}=1$ in $(3.55)$ leads to $(Y(a_2)-\varphi^{-1}(X(1,a_2)))a_2=0$ for $a_2\in A_2$. We assert that
$$
Y(a_2)=\varphi^{-1}(X(1,a_2)).\eqno(3.56)
$$
In fact, replacing $a_2$ by $m+n$ in $(Y(a_2)-\varphi^{-1}(X(1,a_2)))a_2=0$ we get
$$
(Y(n)-\varphi^{-1}(X(1,n)))m+(Y(m)-\varphi^{-1}(X(1,m)))n=0
$$
for all $m, n\in A_2$. In view of Lemma \ref{xxsec3.1}, we have $Y(n)=\varphi^{-1}(X(1,n))$, $Y(m)=\varphi^{-1}(X(1,m))$
for all $m, n\in A_2$. So the assertion holds. Thus $(3.55)$ can be rewritten as
$$
a_1\varphi^{-1}(X(1,a_2))a_{2}=a_2X(a_{1},a_{2})\eqno(3.57)
$$
for all $a_1\in A_1$, $a_2\in A_2$. Substituting $m+n$ for $a_2$ in $(3.57)$, we obtain
$$
a_1\varphi^{-1}(X(1,n))m+a_1\varphi^{-1}(X(1,m))n=mX(a_1,n)+nX(a_1,m)  \eqno(3.58)
$$
for all $a_{1}\in A_{1}, m,n\in A_2$. Replacing $n$ by $na_4$ in $(3.58)$ and substracting the right multiplication of $(3.58)$ by $a_4$, we arrive at
$$
\begin{aligned}
&a_1\varphi^{-1}(X(1,na_4))m-a_{1}\varphi^{-1}(X(1,n))ma_4 \\
&=n[a_4,X(a_1,m)]+mX(a_1,na_4)-mX(a_1,n)a_4
\end{aligned} \eqno(3.59)
$$
for all $a_{1}\in A_{1}$, $a_4\in A_4$, $m,n\in A_2$. Let us set $m=n$ in $(3.59)$ and using $(3.57)$ together, we have
$$
a_1\varphi^{-1}(X(1,ma_{4}))m=mX(a_1,ma_4)+m[a_4,X(a_1,m)]  \eqno(3.60)
$$
for all $a_{1}\in A_{1}, a_4\in A_4, m\in A_2$. Right multipling by $a_4$ in $(3.60)$ and using $(3.57)$, we get $m[X(a_1,ma_4),a_4]=m[X(a_1,m),a_4]a_4$. That is, 
$$
m([X(a_1,ma_4),a_4]-[X(a_1,m),a_4]a_4)=0
$$
for all $a_{1}\in A_{1}$, $a_4\in A_4$, $m\in A_2$. Let us write $P(m)=[X(a_1,ma_4),a_4]-[X(a_1,m),a_4]a_4$,
where $P\colon A_2\longrightarrow A_4$ is an $R$-linear mapping for each $a_{1}\in A_{1}, a_4\in A_4$. Then $mP(m)=0$. Linearizing the relation $mP(m)=0$ leads to $mP(n)+nP(m)=0$ for all $m,n\in A_2$.
In view of Lemma \ref{xxsec3.1}, we obtain $P(m)=0$. That is, 
$$
[X(a_1,ma_4),a_4]=[X(a_1,m),a_4]a_4
$$
for all $a_{1}\in A_{1}, a_4\in A_4, m\in A_2$. Let us choose $b_4\in A_4$ such that $[a_4,b_4]\neq 0$. By lemma \ref{xxsec3.8} we know that $[X(a_1,ma_4),a_4]=[a_4,k_{12}(a_1,ma_4)]\in \mathcal{Z}(A_4)$. Commuting $[X(a_1,ma_4),a_4]=[X(a_1,m),a_4]a_4$ with $b_4$, we have $[X(a_1,m),a_4][a_4,b_4]=0$.
Thus Lemma 3.2 implies that $[X(a_1,m),a_4]=[a_4,k_{12}(a_1,m)]=0$. This shows that $k_{12}(a_1,a_2)\in {\mathcal Z}(A_4)$. In view of $(3.17)$ we know that $[f_{24}(a_2,a_4),a_1]=0$. That is, $f_{24}(a_2,a_4)\in {\mathcal Z}(A_1)$.
\end{proof}

For the proof of Theorem \ref{xxsec3.4} we need to know the form of $f_{12}(a_1,a_2), k_{24}(a_2,a_4),$ $f_{13}(a_1,a_3), k_{34}(a_3,a_4)$
and $f_{14}(a_1,a_4)$.

\begin{lemma}{\rm (\cite[Lemma 3.10]{XiaoWei2})}\label{xxsec3.14}
$f_{12}(a_1,a_2)=\alpha(a_2)a_1+\varphi^{-1}(k_{12}(a_1,a_2))$ and
$k_{24}(a_2,a_4)=\varphi(\alpha(a_2))a_4+\varphi(f_{24}(a_2,a_4))$,
where $\alpha(a_3)=f_{12}(1,a_2)-\varphi^{-1}(k_{12}(1,a_2))$.
\end{lemma}

\begin{lemma}{\rm (\cite[Lemma 3.11]{XiaoWei2})}\label{xxsec3.15}
$f_{13}(a_1,a_3)=\tau(a_3)a_1+\varphi^{-1}(k_{13}(a_1,a_3))$ and
$k_{34}(a_3,a_4)=\varphi(\tau(a_3))a_4+\varphi(f_{34}(a_3,a_4))$,
where $\tau(a_3)=f_{13}(1,a_3)-\varphi^{-1}(k_{13}(1,a_3))$.
\end{lemma}

\begin{lemma}{\rm (\cite[Lemma 3.12]{XiaoWei2})} \label{xxsec3.16}
There exist linear mapping $\gamma\colon A_4\longrightarrow {\mathcal Z}(A_{1})$
and bilinear mapping $\delta\colon A_1\times A_4\rightarrow {\mathcal Z}(A_1)$ such that
$f_{14}(a_{1},a_{4})=\gamma(a_4)a_1+\delta(a_1,a_4)$, for all $a_1\in A_1, a_4\in A_4$.
\end{lemma}

\begin{lemma}{\rm (\cite[Lemma 3.13]{XiaoWei2})}\label{xxsec3.17}
$k_{14}(a_1, a_4)=\gamma^\prime(a_1)a_4+\varphi(\delta(a_1,a_4))$,
where $\gamma^\prime(a_1)=k_{14}(a_1,1^\prime)-\varphi(\delta(a_1,1^\prime))$
for all $a_1\in A_1$, $a_4\in A_4$.
\end{lemma}

\begin{lemma}\label{xxsec3.18}
$g_{24}(a_2,a_4)=a_2(\varepsilon'a_4+\varphi(\gamma(a_4)))$ and
$g_{12}(a_1,a_2)=(\varepsilon a_1+\varphi^{-1}(\gamma'(a_1)))a_2$,
where $\varepsilon'=\kappa-\gamma'(1)$,
$\kappa=\varphi(f_{11}(1,1))-k_{11}(1,1)$,
$\varepsilon=\theta-\gamma(1)$,
$\theta=\varphi^{-1}(k_{11}(1,1))-f_{44}(1,1)$.
\end{lemma}

\begin{proof}
By $(3.34)$ we know that $g_{12}(1,a_2)=f_{11}(1,1)a_2-a_2k_{11}(1,1)$
for all $a_2\in A_2$. Let us set $a_1=1$ in $(3.46)$. We get
$$
g_{24}(a_2,a_4)=a_2(\kappa a_4+\varphi(f_{14}(1,a_4))-k_{14}(1,a_4))
$$
for all $a_2\in A_2$, $a_4\in A_4$, where $\kappa=\varphi(f_{11}(1,1))-k_{11}(1,1)$.
Similarly, using $(3.36)$ and $(3.46)$ we have
$$
g_{12}(a_1,a_2)=(\theta a_1+\varphi^{-1}(k_{14}(a_1,1^\prime))-f_{14}(a_1,1^\prime))a_2
$$
for all $a_1\in A_1$, $a_2\in A_2$, where $\theta=\varphi^{-1}(k_{44}(1^\prime, 1^\prime))-f_{44}(1^\prime, 1^\prime)$.
By the Lemma \ref{xxsec3.16} and Lemma \ref{xxsec3.17} one can conclude the following relations
$$
g_{24}(a_2,a_4)=a_2(\varepsilon'a_4+\varphi(\gamma(a_4))) ,
\quad g_{12}(a_1,a_2)=(\varepsilon a_1+\varphi^{-1}(\gamma'(a_1)))a_2, 
$$
where $\varepsilon=\theta-\gamma(1)$, and $\varepsilon'=\kappa-{\gamma^\prime}^{-1}(1)$.
\end{proof}

\begin{lemma}\label{xxsec3.19}
$h_{13}(a_1,a_3)=a_3\varepsilon a_1+\gamma'(a_1)a_3$ and
$h_{34}(a_3,a_4)=(\varepsilon'a_4+\varphi(\gamma(a_4)))a_3$,
where $\varepsilon'=\kappa-\gamma'(1)$,
$\varepsilon=\theta-\gamma(1)$.
\end{lemma}

\begin{proof}
The proof is similar to that of Lemma \ref{xxsec3.18}.
\end{proof}

{\noindent}{\bf Proof of Theorem 3.4.}
In view of Lemma \ref{xxsec3.18} and Lemma \ref{xxsec3.19}, we obtain 
$$
g_{24}(a_2,a_4)a_3=a_2h_{34}(a_3,a_4).
$$
It follows from the identity $(3.18)$ that
$$
[k_{23}(a_{2},a_{3}),a_{4}]=a_{3}g_{24}(a_{2},a_{4})-h_{34}(a_{3},a_{4})a_{2}.
$$
Combining Lemma \ref{xxsec3.18} with Lemma \ref{xxsec3.19} yields
$$
k_{23}(a_{2},a_{3})-\varepsilon^\prime a_3a_2\in {\mathcal Z}(A_4). \eqno(3.65)
$$
By $(3.19)$ and those similar computational procedures, we arrive at
$$
f_{23}(a_{2},a_{3})-\varepsilon a_2a_3\in {\mathcal Z}(A_1). \eqno(3.66)
$$
Taking $a_1=1$ and $a_4=1^\prime$ into $(3.46)$ and using Lemmas \ref{xxsec3.16}-\ref{xxsec3.18},
we conclude that $\varepsilon a_2=a_2\varepsilon^\prime$ for all $a_2\in A_2$.
Note that $\varepsilon \in {\mathcal Z}(A_1)=\pi_B(\mathcal{Z(G)})$ and $ \varepsilon^\prime\in {\mathcal Z}(A_4)=\pi_B(\mathcal{Z(G)})$. 
Therefore $\varepsilon \oplus \varepsilon^\prime\in \mathcal{Z(G)}$.

By $(3.34)$ and Lemma \ref{xxsec3.18} we have
$$
(f_{11}(a_1,a_1)-\varepsilon a^{2}_1-\varphi^{-1}(\gamma^\prime (a_1))a_1-\varphi^{-1}(k_{11}(a_1,a_1)))a_2=0
$$
for all $a_1\in A_1$, $a_2\in A_2$. Since $A_2=M$ is faithful as a left $A$-module,
$$
f_{11}(a_1,a_1)=\varepsilon a^{2}_1+\varphi^{-1}(\gamma^\prime (a_1))a_1+\varphi^{-1}(k_{11}(a_1,a_1)) \eqno(3.67)
$$
for all $a_1\in A_1$. Similarly,
$$
k_{44}(a_4,a_4)=\varepsilon^\prime  a^{2}_1+\varphi(\gamma(a_4))a_4+\varphi(f_{44}(a_4,a_4)) \eqno(3.68)
$$
for all $a_4\in A_4$.

Finally, let us set $z=\varepsilon\oplus \varepsilon^\prime$ and define the mapping $\mu\colon\mathcal{G}\longrightarrow \mathcal{Z(G)}$ by
$$
\begin{aligned}
&\left[\begin{array}{cc}
a_1 & a_2\\
a_3 & a_4 \end{array}\right]\longmapsto  \\
&\left[\begin{array}{cc}
\varphi^{-1}(\gamma^\prime (a_1))+\gamma(a_4)+\alpha(a_2)+\tau(a_3) & 0\\
0 & \gamma^{'}(a_1)+\varphi(\gamma(a_4)+\alpha(a_2)+\tau(a_3)) \end{array}\right].
\end{aligned}
$$
According to all conclusions derived above, we see that
$$
\begin{aligned}
 \nu(x)\colon &= \mathfrak{T_q}(x)-zx^2-\mu(x)x \\
&\equiv \left[\begin{array}{cc}
f_{23}(a_{2},a_{3})-\varepsilon a_{2}a_{3} & 0\\
0 & k_{23}(a_{2},a_{3})-\varepsilon^\prime a_{3}a_{2} \end{array} \right]\hspace{5pt} ({\rm mod} \mathcal{Z(G)}),
\end{aligned}
$$
where $x=\left[\begin{array}{cc}
a_1 & a_2\\
a_3 & a_4 \end{array}\right]\in \mathcal{G}$.
By $(3.65)$ and $(3.66)$ and the fact $\mathfrak{T_q}$ is centralizing, it follows that
$$
\left[\left[\begin{array}{cc}
f_{23}(a_{2},a_{3})-\varepsilon a_2a_3 & 0\\
0 & k_{23}(a_{2},a_{3})-\varepsilon^\prime  a_3a_2 \end{array}\right],\left[\begin{array}{cc}
a_1 & a_2\\
a_3 & a_4 \end{array}\right]\right]=0.
$$
This implies that $f_{23}(a_2,a_3)-\varepsilon a_2a_3\in {\mathcal
Z}(A_1)=\pi_A(\mathcal{Z(G)})$ and $k_{23}(a_2,a_3)-\varepsilon'
a_3a_2\in {\mathcal Z}(A_4)=\pi_B(\mathcal{Z(G)})$. Moreover, it shows that
$$
(f_{23}(a_{2},a_{3})-\varepsilon a_2a_3)a_2=a_2(k_{23}(a_{2},a_{3})-\varepsilon^\prime a_3a_2)
$$
and
$$
a_3(f_{23}(a_{2},a_{3})-\varepsilon a_2a_3)=(k_{23}(a_{2},a_{3})-\varepsilon^\prime a_3a_2)a_3
$$
for all $a_{2}\in A_2,a_{3}\in A_3$. For convenience, let us write
$\mathfrak{f}(a_{2},a_{3})=f_{23}(a_{2},a_{3})-\varepsilon a_2a_3$, and
$\mathfrak{k}(a_{2},a_{3})=k_{23}(a_{2},a_{3})-\varepsilon^\prime a_3a_2$. Thus
$$
(\mathfrak{f}(a_{2},a_{3})-\varphi^{-1}(\mathfrak{k}(a_{2},a_{3})))a_2=0
$$
for all $a_{2}\in A_2,a_{3}\in A_3$. A linearization of the last relation gives
$$
(\mathfrak{f}(a_{2},a_{3})-\varphi^{-1}(\mathfrak{k}(a_{2},a_{3})))b_2
+(\mathfrak{f}(b_{2},a_{3})-\varphi^{-1}(\mathfrak{k}(b_{2},a_{3})))a_2=0
$$
for all $a_{2}, b_2\in A_2,a_{3}\in A_3$. Note that the hypothesis $A_2=M$ is loyal as an $(A,B)$-bimodule.
It follows from Lemma \ref{xxsec3.1} that
$\mathfrak{f}(a_{2},a_{3})=\varphi^{-1}(\mathfrak{k}(a_{2},a_{3}))$ for all $a_{2}\in A_2,a_{3}\in A_3$.
Hence $\nu$ maps $\mathcal{G}$ into $\mathcal{Z(G)}$. Therefore $\mathfrak{T_q}(x)=zx^2+\mu(x)x+\nu(x)$. This implies that the centralizing trace $\mathfrak{T_q}(x)$ is proper. In the other words, 
centralizing trace of every bilinear mapping in $\mathcal{G}$ is commuting. The proof of Theorem $3.4$ is completed.
\vspace{2mm}

As a direct consequence of Theorem \ref{xxsec3.4} we get

\begin{corollary}{\rm (\cite[Theorem 3.4]{XiaoWei2})}\label{xxsec3.21}
Let $\mathcal{G}=\mathcal{G}(A, M, N, B)$  be a $2$-torsionfree generalized matrix algebra over a
commutative ring $\mathcal{R}$ and ${\mathfrak q}\colon
\mathcal{G}\times \mathcal{G}\longrightarrow \mathcal{G}$ be an
$\mathcal{R}$-bilinear mapping. If
\begin{enumerate}
\item[{\rm(1)}] every commuting linear mapping on $A$ or $B$ is proper;
\item[{\rm(2)}] $\pi_A(\mathcal{Z(G)})=\mathcal{Z}(A)\neq A$ and
$\pi_B(\mathcal{Z(G)})=\mathcal{Z}(B)\neq
B$;
\item[{\rm(3)}] $M$ is loyal,
\end{enumerate}
then every commuting trace ${\mathfrak T}_{\mathfrak q}:
\mathcal{G}\longrightarrow \mathcal{G}$ of ${\mathfrak q}$ is
proper.
\end{corollary}

Following the proof of Theorem $4.3$ and making a slight modification one can show that

\begin{proposition}\label{xxsec3.22}
Let $\mathcal{G}=\mathcal{G}(A, M, N, B)$ be a $2$-torsionfree
generalized matrix algebra over a
commutative ring $\mathcal{R}$, where $B$ is a noncommutative
algebra over $\mathcal{R}$ and both $\mathcal{G}$ and $B$ are
central over $\mathcal{R}$. Suppose that ${\mathfrak q}\colon
\mathcal{G}\times \mathcal{G}\longrightarrow \mathcal{G}$ is an
$\mathcal{R}$-bilinear mapping. If
\begin{enumerate}
\item[{\rm(1)}] every commuting linear mapping of $B$ is proper,
\item[{\rm(2)}] for any $r\in\mathcal{R}$ and $m\in M$, $rm=0$
implies that $r=0$ or $m=0$,
\item[{\rm(3)}] there exist $m_0\in M$ and $b_0\in B$ such that
$m_0b_0$ and $m_0$ are $\mathcal{R}$-linearly independent,
\end{enumerate}
then every centralizing trace $\mathfrak{T_q}\colon
\mathcal{G}\longrightarrow \mathcal{G}$ of $\mathfrak{q}$ is proper.
\end{proposition}

In particular,we also have

\begin{corollary}\label{xxsec3.23}
Let $\mathcal{R}$ be a $2$-torsionfree commutative domain and
${\mathcal M}_n(\mathcal{R})$ be the full matrix algebra over
$\mathcal{R}$. Suppose that ${\mathfrak q}\colon  {\mathcal
M}_n(\mathcal{R})\times {\mathcal M}_n(\mathcal{R})\longrightarrow
{\mathcal M}_n(\mathcal{R})$ is an $\mathcal{R}$-bilinear mapping.
Then every centralizing trace ${\mathfrak T}_{\mathfrak q}\colon
{\mathcal M}_n(\mathcal{R})\longrightarrow {\mathcal
M}_n(\mathcal{R})$ of $\mathfrak{q}$ is proper.
\end{corollary}

\begin{corollary}\label{xxsec3.24}
Let $\mathcal{R}$ be a $2$-torsionfree commutative domain, $V$ be an
$\mathcal{R}$-linear space and $B(\mathcal{R}, V, \gamma)$ be the
inflated algebra of $\mathcal{R}$ along $V$. Suppose that
${\mathfrak q}\colon  B(\mathcal{R}, V, \gamma)$ $\times
B(\mathcal{R}, V, \gamma)\longrightarrow B(\mathcal{R}, V, \gamma)$
is an $\mathcal{R}$-bilinear mapping. Then every centralizing trace
$\mathfrak{T_q}\colon B(\mathcal{R}, V,
\gamma)\longrightarrow B(\mathcal{R}, V, \gamma)$ of $\mathfrak{q}$
is proper.
\end{corollary}

Let us see the centralizing trace of bilinear mappings of several unital algebras with nontrivial idempotents.

\begin{corollary}\label{xxsec3.25}
Let $\mathcal{A}$ be a $2$-torsionfree unital prime algebra over a
commutative ring $\mathcal{R}$. Suppose that $\mathcal{A}$ contains
a nontrivial idempotent $e$ and that $f=1-e$. If
$e\mathcal{Z(A)}e=\mathcal{Z}(e\mathcal{A}e)\neq e\mathcal{A}e$ and
$f\mathcal{Z(A)}f=\mathcal{Z}(f\mathcal{A}f)\neq f\mathcal{A}f$,
then every centralizing trace of an arbitrary bilinear mappings on
$\mathcal{A}$ is proper.
\end{corollary}

\begin{proof}
Let us write $\mathcal{A}$ as a natural generalized matrix algebra
$\left[
\begin{array}
[c]{cc}%
eAe & eAf\\
fAe & fAf\\
\end{array}
\right]$. It is clear that $e\mathcal{A}e$ and $f\mathcal{A}f$ are
prime algebras. By \cite[Theorem 3.2]{Bresar2} it follows that each
commuting additive mapping on $e\mathcal{A}e$ and $f\mathcal{A}f$ is
proper. On the other hand, if $(eae)e\mathcal{A}f(fbf)=0$ holds for
all $a,b\in \mathcal{A}$, then the primeness of $\mathcal{A}$
implies that $eae=0$ or $fbf=0$. This shows $e\mathcal{A}f$ is a
loyal $(e\mathcal{A}e, f\mathcal{A}f)$-bimodule. Applying Theorem
\ref{xxsec3.4} yields that each centralizing trace of an arbitrary
bilinear mappings on $\mathcal{A}$ is proper.
\end{proof}

\begin{corollary}\label{xxsec3.26}
Let $X$ be a Banach space over the real or complex field
$\mathbb{F}$, $\mathcal{B}(X)$ be the algebra of all bounded linear
operators on $X$. Then every ccentralizing trace of an arbitrary
bilinear mapping on $\mathcal{B}(X)$ is proper.
\end{corollary}

\begin{proof}
Note that $\mathcal{B}(X)$ is a centrally closed prime algebra. If
$X$ is infinite dimensional, the result follows from Corollary
\ref{xxsec3.25}. If $X$ is of dimension $n$, then $\mathcal{B}(X)=
{\mathcal M}_n(\mathbb{F})$. In this case the result follows from
Corollary \ref{xxsec3.23}.
\end{proof}

\vspace{2mm}

\section{Lie Triple Isomorphisms on Generalized Matrix Algebras}
\label{xxsec4}

In this section we shall use the main result in Section \ref{xxsec3}
(Theorem \ref{xxsec3.4}) to describe the form of an arbitrary Lie
triple isomorphism of a certain class of generalized matrix algebras
(Theorem \ref{xxsec4.4}). As applications of Theorem \ref{xxsec4.4},
we characterize Lie isomorphisms of certain generalized matrix
algebras. The involved algebras include upper triangular matrix
algebras, nest algebras, full matrix algebras, inflated algebras,
prime algebras with nontrivial idempotents.

\begin{lemma}\label{xxsec4.1}
Let $\mathcal{G}=\mathcal{G}(A, M, N, B)$ be a $2$-torsionfree generalized matrix algebra over a
commutative ring $\mathcal{R}$. Then $\mathcal{G}$ does not contain nonzero central Jordan ideals.
\end{lemma}

\begin{proof}
Let $\mathcal{I}$ be a central Jordan ideal of $\mathcal{G}=\mathcal{G}(A, M, N, B)$.
Suppose $\alpha\oplus \varphi(\alpha)\in \mathcal{I}$. Thus
$$
\left[\begin{array}{cc}
\alpha & 0\\
   0   & \varphi(\alpha)\end{array}\right]\circ \left[\begin{array}{cc}
0 & m\\
 0 & 0\end{array}\right]=\left[\begin{array}{cc}
0 & \alpha m+m\varphi(\alpha)\\
 0 & 0\end{array}\right]
$$
for all $m\in \mathcal{G}$. Thus yields $2\alpha M=0$
and hence $\alpha=0=\alpha\oplus \varphi(\alpha)$.
\end{proof}

\begin{lemma}\cite[Lemma 4.1]{XiaoWei2}\label{xxsec4.2}
Let $\mathcal{G}=\mathcal{G}(A, M, N, B)$ be a $2$-torsionfree generalized matrix algebra over a
commutative ring $\mathcal{R}$. Then $\mathcal{G}$ satisfies the
polynomial identity $[[x^2,y],[x,y]]$ if and only if both $A$ and
$B$ are commutative.
\end{lemma}

\begin{proposition}\label{xxsec4.3}
Let $\mathcal{G}=\mathcal{G}(A, M, N, B)$ and $\mathcal{G}^\prime=\mathcal{G}^\prime(A^\prime, M^\prime, N^\prime, B^\prime)$ be two generalized matrix algebras over a commutative ring
$\mathcal{R}$ with $\frac{1}{2}\in \mathcal{R}$ and let
$\mathfrak{l}:\mathcal{G}\longrightarrow \mathcal{G}^\prime$ be a
Lie triple isomorphism. If
\begin{enumerate}
\item[(1)] each centralizing trace of arbitrary bilinear mapping on $\mathcal{G}^\prime$ is proper,
\item[(2)] at least one of $A, B$ and at least one of  $A^\prime, B^\prime$ are noncommutative,
\item[(3)] $M^\prime$ is loyal,
\end{enumerate}
then $\mathfrak{l}=\pm \mathfrak{m}+\mathfrak{n}$, where
$\mathfrak{m}\colon \mathcal{G}\longrightarrow \mathcal{G}^\prime$ is a Jordan
homomorphism, $\mathfrak{m}$ is one-to-one, and $\mathfrak{n}\colon
\mathcal{G}\longrightarrow \mathcal{Z(G^\prime)}$ is a linear
mapping vanishing on each second commutator. Moreover, if
$\mathcal{G}^\prime$ is central over $\mathcal{R}$, then
$\mathfrak{m}$ is onto.
\end{proposition}

\begin{proof}
For arbitrary $x,z\in \mathcal{G}$, it is easy to see that
$\mathfrak{l}$ satisfies
$[[\mathfrak{l}(x^{2}),\mathfrak{l}(x)],\mathfrak{l}(z)]=\mathfrak{l}([[x^{2},x],z])=0$.
Since $\mathfrak{l}$ is onto,
$[\mathfrak{l}(x^{2}),\mathfrak{l}(x)]\in \mathcal{Z(G^\prime)}$ for
all $x\in \mathcal{G}$. Replacing $x$ by $\mathfrak{l}^{-1}(y)$, we
get $[\mathfrak{l}(\mathfrak{l}^{-1}(y)^{2}),y]\in
\mathcal{Z(G^\prime)}$ for all $y\in \mathcal{G^\prime}$. This implies
that the mapping
$\mathfrak{T_q}(y)=\mathfrak{l}(\mathfrak{l}^{-1}(y)^{2})$ is
centralizing. Since $\mathfrak{T_q}$ is also a trace of the bilinear
mapping $\mathfrak{q}\colon \mathcal{G'}\times \mathcal{G^\prime}
\longrightarrow \mathcal{G^\prime}$,
$\mathfrak{q}(y,z)=\mathfrak{l}(\mathfrak{l}^{-1}(y)\mathfrak{l}^{-1}(z))$,
by the hypothesis $(1)$ there exist $\lambda\in
\mathcal{Z(G^\prime)}$, a linear mapping $\mu_{1}:
\mathcal{G^\prime}\longrightarrow \mathcal{Z(G^\prime)}$, and a
trace $\nu_{1}: \mathcal{G^\prime}\longrightarrow
\mathcal{Z(G^\prime)}$ of a bilinear mapping such that
$$
\mathfrak{l}(\mathfrak{l}^{-1}(y)^{2})=\lambda
y^{2}+\mu_{1}(y)y+\nu_{1}(y) \eqno(4.1)
$$
for all $y\in \mathcal{G'}$. Let us write $\mu=\mu_1\mathfrak{l}$ and
$\nu=\nu_1\mathfrak{l}$. Then $\mu$ and $\nu$ are mappings of
$\mathcal{G}$ into $\mathcal{Z(G^\prime)}$ and $\mu$ is linear.
Hence $(4.1)$ can be rewritten as
$$
\mathfrak{l}(x^{2})=\lambda\mathfrak{l}(x)^{2}+\mu(x)\mathfrak{l}(x)+\nu(x)
\eqno(4.2)
$$
for all $x\in \mathcal{G}$. We conclude that $\lambda\neq 0$.
Otherwise, we have $\mathfrak{l}(x^{2})-\mu(x)\mathfrak{l}(x)\in
\mathcal{Z(G^\prime)}$ by $(4.2)$ and hence
$$
\begin{aligned}
\mathfrak{l}([[x^2,y],[x,y]])
&=[[\mathfrak{l}(x^2),\mathfrak{l}(y)],\mathfrak{l}([x,y])]\\
&=[[\mu(x)\mathfrak{l}(x),\mathfrak{l}(y)],\mathfrak{l}([x,y])] \\
&=\mu(x)[[\mathfrak{l}(x),\mathfrak{l}(y)],\mathfrak{l}([x,y])]\\
&=\mu(x)\mathfrak{l}([[x,y],[x,y]])\\
&=0
\end{aligned}
$$
for all $x,y\in \mathcal{G}$. Consequently, $[[x^2,y],[x,y]]=0$ for
all $x,y\in \mathcal{G}$. According to our assumption this
contradicts with {\rm (\cite[Lemma 2.7]{BenkovicEremita1})}. Thus
$\lambda\neq 0$.

Now we define a linear mapping $\mathfrak{m}\colon
\mathcal{G}\longrightarrow \mathcal{G^{'}}$ by
$$
\mathfrak{m}(x)=\lambda\mathfrak{l}(x)+\frac{1}{2}\mu(x) \eqno(4.3)
$$
for the $x\in \mathcal{G}$. Of course, $\mathfrak{m}$ is a linear
mapping. Our goal is to show that $\mathfrak{m}$ is a Jordan
homomorphism. In view of $(4.2)$ and $(4.3)$, we have
$$
\mathfrak{m}(x^{2})=\lambda\mathfrak{l}(x^{2})+\frac{1}{2}\mu(x)=
\lambda^{2}\mathfrak{l}(x)^{2}+\lambda\mu(x)\mathfrak{l}(x)+\lambda\nu(x)+\frac{1}{2}\mu(x^{2}),
$$
while
$$
\mathfrak{m}(x)^{2}=(\lambda\mathfrak{l}(x)+\frac{1}{2}\mu(x))^{2}
=\lambda^{2}\mathfrak{l}(x)^{2}+\lambda\mu(x)\mathfrak{l}(x)+\frac{1}{4}\mu(x)^{2}.
$$
Comparing the above two identities we get
$$
\mathfrak{m}(x^{2})-\mathfrak{m}(x)^{2}\in \mathcal{Z(G^\prime)}
\eqno(4.4)
$$
for all $x\in \mathcal{G}$. Linearizing $(4.4)$ we obtain
$$
\mathfrak{m}(x\circ y)-\mathfrak{m}(x)\circ \mathfrak{m}(y)\in
\mathcal{Z(G^\prime)}
$$
for all $x, y\in \mathcal{G}$. Define the mapping $\varepsilon:
\mathcal{G}\times \mathcal{G}\longrightarrow \mathcal{Z(G^\prime)}$ by
$$
\varepsilon(x,y)=\mathfrak{m}(x\circ y)-\mathfrak{m}(x)\circ
\mathfrak{m}(y). \eqno(4.5)
$$
Clearly, $\varepsilon$ is a symmetric bilinear mapping. Of course,
$\mathfrak{m}$ is a Jordan homomorphism if and only if
$\varepsilon(x, y)=0$ for all $x, y\in \mathcal{G}$. For any $x,y\in
\mathcal{G}$, let us put $W=\mathfrak{m}(x\circ (x\circ y))$. By
$(4.5)$ we have
$$
\begin{aligned}
W&=\mathfrak{m}(x)\mathfrak{m}(x\circ y)+\mathfrak{m}(x\circ y)\mathfrak{m}(x)+\varepsilon(x,x\circ y) \\
&=\mathfrak{m}(x)\{\mathfrak{m}(x)\circ\mathfrak{m}(y)+\varepsilon(x,y)\}
+[\mathfrak{m}(x)\circ\mathfrak{m}(y)+\varepsilon(x,y)]\mathfrak{m}(x)+\varepsilon(x,x\circ y) \\
&=\mathfrak{m}(x)^{2}\mathfrak{m}(y)+2\mathfrak{m}(x)\mathfrak{m}(y)\mathfrak{m}(x)
+\mathfrak{m}(y)\mathfrak{m}(x)^{2}+2\varepsilon(x,y)\mathfrak{m}(x)+\varepsilon(x,x\circ
y).
\end{aligned}
$$
On the other hand
$$
\begin{aligned}
W&=2\mathfrak{m}(xyx)+\mathfrak{m}(x^{2}\circ y) \\
&=2\mathfrak{m}(xyx)+\mathfrak{m}(x^{2})\circ \mathfrak{m}(y)+\varepsilon(x^{2},y) \\
&=2\mathfrak{m}(xyx)+[\mathfrak{m}(x^{2})+\frac{1}{2}\varepsilon(x,x)]\mathfrak{m}(y) \\
&\quad +\mathfrak{m}(y)[\mathfrak{m}(x^{2})+\frac{1}{2}\varepsilon(x,x)]+\varepsilon(x^{2},y) \\
&=2\mathfrak{m}(xyx)+\mathfrak{m}(x)^{2}\mathfrak{m}(y)+\mathfrak{m}(y)\mathfrak{m}(x)^{2} \\
&\quad +\varepsilon(x,x)\mathfrak{m}(y)+\varepsilon(x^{2},y).
\end{aligned}
$$
Comparing the above two relations gives
$$
\begin{aligned}
\mathfrak{m}(xyx)
&=\mathfrak{m}(x)\mathfrak{m}(y)\mathfrak{m}(x)+\varepsilon(x,y)\mathfrak{m}(x)
-\frac{1}{2}\varepsilon(x,x)\mathfrak{m}(y) \\
&\quad +\frac{1}{2}\varepsilon(x,x\circ
y)\mathfrak{m}(y)-\frac{1}{2}\varepsilon(x^{2},y).
\end{aligned}\eqno(4.6)
$$
By completing linearization of $(4.6)$ we obtain
$$
\begin{aligned}
\mathfrak{m}(xyz+zyx)
&=\mathfrak{m}(x)\mathfrak{m}(y)\mathfrak{m}(z)
+\mathfrak{m}(z)\mathfrak{m}(y)\mathfrak{m}(x)+\varepsilon(x,y)\mathfrak{m}(z) \\
&\quad +\varepsilon(z,y)\mathfrak{m}(x)-\varepsilon(x,z)\mathfrak{m}(y)
+\frac{1}{2}\varepsilon(x,z\circ y) \\
&\quad +\frac{1}{2}\varepsilon(z,x\circ
y)-\frac{1}{2}\varepsilon(x\circ z,y).
\end{aligned}\eqno(4.7)
$$
Let us consider $U=\mathfrak{m}(xyx^{2}+x^{2}yx)$. By $(4.7)$ we
know that
$$
\begin{aligned}
U&=\mathfrak{m}(x)\mathfrak{m}(y)\mathfrak{m}(x^{2})
+\mathfrak{m}(x^{2})\mathfrak{m}(y)\mathfrak{m}(x)+\varepsilon(x,y)\mathfrak{m}(x^{2}) \\
&\quad +\varepsilon(x^{2},y)\mathfrak{m}(x)-\varepsilon(x,x^{2})\mathfrak{m}(y)
+\frac{1}{2}\varepsilon(x,x^{2}\circ y) \\
&\quad +\frac{1}{2}\varepsilon(x^{2},x\circ y)-\frac{1}{2}\varepsilon(x^{3},y).
\end{aligned}
$$
Since
$\mathfrak{m}(x^{2})=\mathfrak{m}(x)^{2}+\frac{1}{2}\varepsilon(x,x)$,
we get
$$
\begin{aligned}
U&=\mathfrak{m}(x)\mathfrak{m}(y)\mathfrak{m}(x)^{2}
+\mathfrak{m}(x)^{2}\mathfrak{m}(y)\mathfrak{m}(x)+\varepsilon(x,x)\mathfrak{m}(x)\mathfrak{m}(y) \\
&\quad +\frac{1}{2}\varepsilon(x,x)\mathfrak{m}(y)\mathfrak{m}(x)
+\varepsilon(x,y)\mathfrak{m}(x)^{2}+\varepsilon(x^{2},y)\mathfrak{m}(x) \\
&\quad -\varepsilon(x,x^{2})\mathfrak{m}(y)+\frac{1}{2}\varepsilon(x,y)\varepsilon(x,x)
+\frac{1}{2}\varepsilon(x,x^{2}\circ y) \\
&\quad +\frac{1}{2}\varepsilon(x^{2},x\circ y)-\varepsilon(x^{3},y).
\end{aligned}
$$
On the other hand, using $(4.5)$ and $(4.6)$ we have
$$
\begin{aligned}
U&=\mathfrak{m}((xyx)\circ x) \\
&=\mathfrak{m}(xyx)\circ \mathfrak{m}(x)+\varepsilon(xyx,x) \\
&=\mathfrak{m}(x)\mathfrak{m}(y)\mathfrak{m}(x)^{2}
+\mathfrak{m}(x)^{2}\mathfrak{m}(y)\mathfrak{m}(x)+2\varepsilon(x,y)\mathfrak{m}(x)^{2} \\
&\quad -\frac{1}{2}\varepsilon(x,x)(\mathfrak{m}(y)\circ
\mathfrak{m}(x))+\varepsilon(x,x\circ
y)\mathfrak{m}(x)-\varepsilon(x^{2},y)\mathfrak{m}(x)+\varepsilon(xyx,x).
\end{aligned}
$$
Comparing the above two relations yields
$$
\begin{aligned}
&\varepsilon(x,x)\mathfrak{m}(x)\circ \mathfrak{m}(y)
-\varepsilon(x,y)\mathfrak{m}(x)^{2}-\varepsilon(x,x^{2})\mathfrak{m}(y) \\
&\quad +(2\varepsilon(x^{2},y)-\varepsilon(x,x\circ
y))\mathfrak{m}(x)\in \mathcal{Z(G^\prime)}
\end{aligned}\eqno(4.8)
$$
for all $x, y\in \mathcal{G}$. In particular, if $x=y$, we obtain
$$
\varepsilon(x,x)\mathfrak{m}(x)^{2}-\varepsilon(x,x^{2})\mathfrak{m}(x)\in
\mathcal{Z(G^\prime)} \eqno(4.9)
$$
for all $x\in \mathcal{G}$. Therefore
$$
\varepsilon(x,x)[[\mathfrak{m}(x)^{2},u],[\mathfrak{m}(x),u]]=0
$$
for all $x\in \mathcal{G}, u\in \mathcal{G'}$, which can be in view
of $(4.3)$ rewritten as
$$
\lambda^{3}\varepsilon(x,x)[[\mathfrak{l}(x)^{2},u],[\mathfrak{l}(x),u]]=0.
$$
We may assume that $A'$ is noncommutative. Pick $a_1,a_2\in A'$ such
that $a_1[a_1,a_2]a_1\neq 0$ (see the proof of \cite[Lemma
2.7]{BenkovicEremita1}). Setting
$$
\mathfrak{l}(x_{0})=\left[\begin{array}{cc}
a_{1} & 0\\
& 0\end{array}\right]\quad \text{and}\quad
u_{0}=\left[\begin{array}{cc}
a_{2} & m\\
& 0\end{array}\right]
$$
for some $x_{0}\in \mathcal{G}$ and an arbitrary $m\in  M'$ in the
relation
$\lambda^{3}\varepsilon(x,x)[[\mathfrak{l}(x)^{2},u],[\mathfrak{l}(x),u]]=0$,
we arrive at
$$
\pi_{A'}(\lambda^{3}\varepsilon(x_{0},x_{0}))a_{1}[a_{1},a_{2}]a_{1}m=0
$$
for all $m\in M^\prime$. By the loyality of $M^\prime$ it follows
that
$\pi_{A'}(\lambda^{3}\varepsilon(x_{0},x_{0}))a_{1}[a_{1},a_{2}]a_{1}=0$.
Hence $\pi_{A'}(\lambda^{3}\varepsilon(x_{0},x_{0}))=0$ by Lemma
\ref{xxsec3.2}. This shows that
$\lambda^{3}\varepsilon(x_{0},x_{0})=0$. Since $\lambda\neq 0$,
$\varepsilon(x_{0},x_{0})=0$ by Lemma \ref{xxsec3.3}. Taking $\varepsilon(x_{0},x_{0})=0$
into $(4.9)$ and then making the commutator with $u_0$ we can obtain $\varepsilon(x_{0},x_{0}^{2})=0$.
From $(4.8)$ we get
$$
\varepsilon(x_{0},y)\mathfrak{m}(x_{0})^2
+[-2\varepsilon(x_0^2, y)+\varepsilon(x_{0},x_{0}\circ y)]\mathfrak{m}(x_{0})\in
\mathcal{Z(G^\prime)} \eqno(4.10)
$$
for all $y\in \mathcal{G}$. Let us commute the above relation with $u_0$ and then
with $[\mathfrak{m}(x_{0}), u_0]$ in order.
We will eventually observe that $\varepsilon(x_{0},y)=0$ for all $y\in \mathcal{G}$.
Then the equation $(4.10)$ shows that $2\varepsilon(x_0^2, y)\mathfrak{m}(x_{0})\in
\mathcal{Z(T^\prime)}$ for all $y\in \mathcal{G}$.
Therefore $\varepsilon(x_0^2, y)[\mathfrak{m}(x_{0}), u_0]=0$
and hence $\varepsilon(x_0^2, y)=0$ for all $y\in \mathcal{G}$.

Next, we assert that $\varepsilon(x,y)=0$. Substituting $x_{0}+y$ for
$x$ by in $(4.9)$ and using the fact $\varepsilon(x_{0},y)=0$, we have
$$
\begin{aligned}
&\varepsilon(y,y)\mathfrak{m}(x_{0})^{2}+\varepsilon(y,y)\mathfrak{m}(x_{0})\circ \mathfrak{m}(y)
-\varepsilon(y,(x_{0}+y)^{2})\mathfrak{m}(x_{0}) \\
&-\varepsilon(y,x_{0}\circ y)\mathfrak{m}(y)\in \mathcal{Z(G^\prime)}.
\end{aligned}
$$
On the other hand, replacing $x$ by $-x_{0}+y$ in $(4.9)$ we get
$$
\begin{aligned}
&\varepsilon(y,y)\mathfrak{m}(x_{0})^{2}-\varepsilon(y,y)\mathfrak{m}(x_{0})\circ \mathfrak{m}(y)
+\varepsilon(y,(x_{0}-y)^{2})\mathfrak{m}(x_{0}) \\
&+\varepsilon(y,x_{0}\circ y)\mathfrak{m}(y)\in \mathcal{Z(G^\prime)}.
\end{aligned}
$$
Comparing the two relations it follows that
$$
\varepsilon(y,y)\mathfrak{m}(x_{0})^{2}
-\varepsilon(y,x_{0}\circ
y)\mathfrak{m}(x_{0})\in \mathcal{Z(G^\prime)}.
$$
Commuting with $u_{0}$ and then with $[\mathfrak{m}(x_{0}),u_{0}]$,
in view of $(4.3)$ the above relation becomes
$$
\varepsilon(y,y)[[\mathfrak{l}(x_{0})^{2},u_{0}],[\mathfrak{l}(x_{0}),u_{0}]]=0.
\eqno(4.11)
$$
Furthermore, $\varepsilon(y,y)=0$ for all $y\in \mathcal{G}$.
Hence $\varepsilon=0$ by the symmetry of $\varepsilon$. This shows that $\mathfrak{m}$ is a Jordan
homomorphism.

We claim that $\lambda=\pm 1$. By $(4.3)$ it follows that
$$
\begin{aligned}
\lambda^{2}\mathfrak{m}([[x,y],z])
&=\lambda^{3}\mathfrak{l}([[x,y],z])+\frac{1}{2}\lambda^{2}\mu([[x,y],z]) \\
&=[[\mathfrak{m}(x),\mathfrak{m}(y)],\mathfrak{m}(z)]+\frac{1}{2}\lambda^{2}\mu([[x,y],z])
\end{aligned}
$$
for all $x,y,z\in \mathcal{G}$. Moreover, we get
$$
\lambda^{2}\mathfrak{m}([[x,y],z])-[[\mathfrak{m}(x),\mathfrak{m}(y)],
\mathfrak{m}(z)]\in \mathcal{Z(G^\prime)} \eqno(4.12)
$$
for all $x,y,z\in \mathcal{G}$. Considering $(4.12)$ and using the
facts $\mathfrak{m}(x\circ y)=\mathfrak{m}(x)\circ \mathfrak{m}(y)$ and
$[[x,y],z]=x\circ (y\circ z)-y\circ (x\circ z)$, we conclude that
$$
(\lambda^{2}-1)[[\mathfrak{m}(x),\mathfrak{m}(y)],\mathfrak{m}(z)]\in
\mathcal{Z(G^\prime)}.
$$
for all $x,y,z\in \mathcal{G}$. By $(4.3)$ we know that
$\lambda^{3}(\lambda^{2}-1)\mathfrak{l}([[x,y],z])\in
\mathcal{Z(G^\prime)}$. Since $x,y,z$ are arbitrary elements in
$\mathcal{G}$ and $\mathfrak{l}$ is bijective, we eventually obtain
$\lambda^{3}(\lambda^{2}-1)=0$. Since $\lambda\neq 0$, we get
$\lambda=\pm 1$.

Let us put $\mathfrak{n}(x)=-\frac{1}{2}\mu(x)$. When $\lambda=1$,
then $\mathfrak{l}=\mathfrak{m}+\mathfrak{n}$.
It is easy to verify that $\mathfrak{n}([[x, y], z])=0$ for all $x, y, z\in \mathcal{G}$.
Note that $\mathfrak{m}$ is a Jordan homomorphism from $\mathcal{G}$ into $\mathcal{G}^\prime$
and hence is a Lie triple homomorphism from $\mathcal{G}$ into $\mathcal{G}^\prime$.
When $\lambda=-1$, then $\mathfrak{n}=\mathfrak{l}+\mathfrak{m}$ is a Lie triple homomorphism from
$\mathcal{G}$ into $\mathcal{Z(G^\prime)}$.
Therefore $\mathfrak{n}([[x, y], z])=0$ for all $x, y, z\in \mathcal{G}$.

We have to prove that $\mathfrak{m}$ is one-to-one. Suppose that
$\mathfrak{m}(w)=0$ for some $w\in \mathcal{G}$. Then
$\mathfrak{l}(w)\in \mathcal{Z(G^\prime)}$ and hence $w\in
\mathcal{Z(G)}$. This implies that ${\rm ker}(\mathfrak{m})\subseteq
\mathcal{Z(G)}$. That is, ${\rm ker}(\mathfrak{m})$ is a Jordan
ideal of $\mathcal{Z(G)}$. However, by Lemma \ref{xxsec4.1} it
follows that ${\rm ker}(\mathfrak{m})=0$.

It remains to prove that $\mathfrak{m}$ is onto in case
$\mathcal{G^\prime}$ is central over $\mathcal{R}$. Let us first
show that $\mathfrak{m}(1)=1^\prime$. Since $\mathfrak{l}$ is a Lie
triple isomorphism, we have $\mathfrak{l}(1)\in
\mathcal{Z(G^\prime)}$ and hence
$\mathfrak{m}(1)=\mathfrak{l}(1)-\mathfrak{n}(1)\in \mathcal{Z(G^\prime)}$.
Note that $\mathfrak{m}$ is a Jordan homomorphism. We see that
$2\mathfrak{m}(x)=\mathfrak{m}(x\circ
1)=2\mathfrak{m}(x)\mathfrak{m}(1)$. Since $\frac{1}{2}\in
\mathcal{R}$, $(\mathfrak{m}(1)-1^{'})\mathfrak{m}(x)=0$, which can
be rewritten as $(\mathfrak{m}(1)-1^{'})\mathfrak{l}(x)\in
\mathcal{Z(G^\prime)}$. Then
$(\mathfrak{m}(1)-1')[\mathfrak{l}(x),s]=0$ for all $s\in
\mathcal{G'}$. Therefore
$(\mathfrak{m}(1)-1')[\mathcal{G'},\mathcal{G'}]=0$. Consequently,
$\pi_{A'}(\mathfrak{m}(1)-1')[A',A']=0$. This implies that
$\pi_{A'}(\mathfrak{m}(1)-1')=0$ and so $\mathfrak{m}(1)=1'$.
Obviously, we may write $\mathfrak{n}(x)=f(x)1^\prime$
for some linear mapping $f: \mathcal{G}\longrightarrow \mathcal{R}$.
Since $\mathfrak{m}$ is $\mathcal{R}$-linear, we obtain that
$\mathfrak{l}(x)=\pm\mathfrak{m}(x)+f(x)1'=\mathfrak{m}(\pm
x+f(x)1)$ for all $x\in \mathcal{G}$. Consequently $\mathfrak{m}$
is onto, since $\mathfrak{l}$ is bijective. The proof of the theorem is thus completed.
\end{proof}

\begin{theorem}\label{xxsec4.4}
Let $\mathcal{G}=\mathcal{G}(A, M, N, B)$ and  $\mathcal{G}^\prime=\mathcal{G}^\prime(A^\prime, M^\prime, N^\prime, B^\prime)$ be two generalized matrix algebras over a commutative ring $\mathcal{R}$ with $\frac{1}{2}\in \mathcal{R}$. Let $\theta\colon \mathcal{G}\longrightarrow \mathcal{G^\prime}$ be a Lie triple isomorphism. If
\begin{enumerate}
\item[(1)] each commuting linear mapping on $A^\prime$ or $B^\prime$ is proper,
\item[(2)] $\pi_{A^\prime}(\mathcal{Z(G^\prime)})={\mathcal Z}(A^\prime) \neq A^\prime$ and $\pi_{B^\prime}(\mathcal{Z(G^\prime)})={\mathcal Z}(B^\prime)\neq B^\prime$,
\item[(3)] either $A$ or $B$ is noncommutative,
\item[(4)] $M^\prime$ is loyal,
\end{enumerate}
then $\theta=\pm\mathfrak{m}+\mathfrak{n}$,
where $\mathfrak{m}:\mathcal{G}\rightarrow \mathcal{G}^{'}$ is Jordan homomorphism,
$\mathfrak{m}$ is one-to-one,and $\mathfrak{n}\colon \mathcal{G}\longrightarrow {\mathcal Z}(\mathcal{G}^{'})$
is a linear mapping sending second commutators to zero.
Moreover, if $\mathcal{G'}$ is central over $\mathcal{R}$, then $\mathfrak{m}$ is surjective.
\end{theorem}

\begin{proof}
It follows from Theorem $3.4$ and Proposition$4.3$
\end{proof}

\section{Topic for Future Research}\label{xxsec5}

Although the main aim of this paper is to describe 
the form of Lie triple isomorphisms on generalized matrix algebras, 
the investgation of various additive mappings (associative-type, Jordan-type or Lie-type) on generalized matrix algebras also have a great interest and should be further paid much attention. The study of additive mappings on generalized matrix algebras is shedding light on the investigation of functional identities in the setting of such kind of algebras.
In the light of the motivation and contents of this article, we
will propose several topics with high potential and with merit for
future research in this field. 

The theory of functional identitieswas initiated by Bre\v{s}ar at the beginning of 90's in last century and it was greatly
developed by Beidar, Bre\v{s}ar, Chebotar, and Martindale. We refer the reader to the monograph \cite{BresarChebotarMartindale} for a full account on the theory of functional identities. Let $\mathcal{A}$ be an associative ring. Let $F_1, F_2, G_1, G_2$ be mappings from $\mathcal{A}$ into itself such that
$$
F_1(x)y+F_2(y)x+xG_2(y)+yG_1(x)=0 \eqno(5.1)
$$
for all $x, y\in \mathcal{A}$. This is a basic functional identity, which was one of the first functional
identities studied in prime algebras. The mappings $F_1, F_2, G_1$ and $G_2$ are looked on as unknowns
and the main prupose is to describe the form of these mappings. Functional identity $(5.1)$ is also
closely related to commuting mappings, which is due to the fact each commuting additive mapping $F$ of $\mathcal{A}$ satisfies the relation
$$
F(x)y+F(y)x-xF(y)-yF(x) = 0
$$
for all $x, y\in \mathcal{A}$. This relation is just one special case of $(5.1)$. Centralizing mappings and commuting mappings can be considered as the most basic and important examples of functional identities. However, the general theory of functional identities, which
was developed in \cite{BresarChebotarMartindale}, cannot be applied in the context of triangular rings, since these rings are not $d$-free. Nevertheless, Beidar, Bre\v{s}ar and Chebotar investigated certain functional
identities on upper triangular matrix algebras \cite{BeidarBresarChebotar}. Moreover, Cheung described
the form of commuting linear maps for a certain class of triangular algebras \cite{Cheung2}. Later, several
problems on certain types of mappings on triangular rings and algebras have been studied,
where some special examples of functional identities appear. Zhang and his students \cite{ZhangFengLiWu} studied the functional identity $(5.1)$ in the context of nest algebras. In two recent articles \cite{Eremita1, Eremita2}, Eremita considered functional identity $(5.1)$ in
triangular algebras. He managed to describe the form of additive mappings $F_1, F_2, G_1, G_2\colon \mathcal{T}\longrightarrow \mathcal{T}$
satisfying $(5.1)$ if a triangular ring $\mathcal{T}$ satisfies certain conditions. Moreover, the notion of the
maximal left ring of quotients, which plays an important role in the study of functional identities on (semi-)prime rings, is used to characterize those additive mappings $F_1, F_2, G_1, G_2$ \cite{Eremita2}. It is predictable  that Eremita's approach, which is based on the notion of the maximal left ring
of quotients, enable us to generalize and unify a number of known results regarding mappings of triangular algebras and generalized matrix algebras.

The functional identities in triangular algebras and full matrix algebras were
already studied in \cite{BeidarBresarChebotar, BresarSpenko, Eremita1, Eremita2}. 
One would expect that the next step is to investigate
functional identities of generalized matrix algebras. The notion of generalized matrix algebras
efficiently unifies triangular algebras and full matrix algebras
together. The eventual goal of our systematic work is to deal
with all questions related to additive (or multiplicative) mappings of triangular
algebras and full matrix algebras under a unified frame, which is
the desirable generalized matrix algebras frame.

\begin{question}\label{xxsec5.1}
Let $\mathcal{G}=\mathcal{G}(A, M, N, B)$ be a generalized matrix algebra over a commutative ring $\mathcal{R}$ and let $F_1, F_2, G_1, G_2$ be mappings from $\mathcal{G}$ into itself such that 
$$
F_1(x)y+F_2(y)x+xG_2(y)+yG_1(x)=0
$$
for all $x,y\in \mathcal{G}$. Describe the forms of $F_1, F_2, G_1, G_2$ satisfying the above condition.
\end{question}

Although people embark on studying functional identities in triangular algebras and full matrix algebras, the functional identities with additional structure has not been treated yet. For instance, the functional identities with automorphisms and derivations in generalized matrix algebras are worthy to be considered further. 

\begin{proposition}{\rm(\cite[Proposition 4.2]{LiWei1})}\label{xxsec5.2}
Let $\mathcal{G}=\mathcal{G}(A, M, N, B)$ be a generalized matrix algebra over a commutative ring $\mathcal{R}$. An $\mathcal{R}$-linear mapping $\Theta$ is a derivation of $\mathcal{G}$ if and only if
$\Theta$ has the form
$$ 
\Theta\left(\left[
\begin{array}
[c]{cc}%
a & m\\
n & b\\
\end{array}
\right]\right) 
=\left[
\begin{array}
[c]{cc}%
\delta_1(a)-mn_0-m_0n & am_0-m_0b+\tau_2(m)\\
n_0a-bn_0+\nu_3(n) & n_0m+nm_0+\mu_4(b)\\
\end{array}
\right]
$$
where $m_0\in M, n_0\in N$ and
$$
\begin{aligned} \delta_1:& A \longrightarrow A, &
 \tau_2: & M\longrightarrow M, & \tau_3: & N\longrightarrow M,\\
\nu_2: & M\longrightarrow N, & \nu_3: & N\longrightarrow N , &
\mu_4: & B\longrightarrow B
\end{aligned}
$$
are all $\mathcal{R}$-linear mappings satisfying the following
conditions:
\begin{enumerate}
\item[{\rm(1)}] $\delta_1$ is a derivation of $A$ with
$\delta_1(mn)=\tau_2(m)n+m\nu_3(n);$

\item[{\rm(2)}] $\mu_4$ is a derivation of $B$ with
$\mu_4(nm)=n\tau_2(m)+\nu_3(n)m;$

\item[{\rm(3)}] $\tau_2(am)=a\tau_{2}(m)+\delta_1(a)m$ and
$\tau_2(mb)=\tau_2(m)b+m\mu_4(b);$

\item[{\rm(4)}] $\nu_3(na)=\nu_3(n)a+n\delta_1(a)$ and
$\nu_3(bn)=b\nu_3(n)+\mu_4(b)n.$
\end{enumerate}
\end{proposition}

In view of Proposition \ref{xxsec5.2} we have

\begin{question}\label{xxsec5.3}
Let $\mathcal{G}=\mathcal{G}(A, M, N, B)$ be a generalized matrix algebra over a commutative ring $\mathcal{R}$. Let $F_1, F_2, G_1, G_2$ be mappings from $\mathcal{G}$ into itself and $\Theta_1, \Theta_2, \Delta_1, \Delta_2$ be derivations from $\mathcal{G}$ into itself such that  
$$
F_1(x)\Theta_2(y)+F_2(y)\Theta_1(x)+\Delta_1(x)G_2(y)+\Delta_2(y)G_1(x)=0
$$
for all $x,y\in \mathcal{G}$. What can we say about the forms of $F_1, F_2, G_1, G_2$ and those of $\Theta_1, \Theta_2, \Delta_1, \Delta_2$ satisfying the above condition ?
\end{question}

Let us next consider the functional identities with automorphisms in generalized matrix algebras. 

\begin{proposition}{\rm(\cite[Theorem 3.6]{BobocDascalescuWyk})}\label{xxsec5.4}
Let $\mathcal{G}=\mathcal{G}(A, M, N, B)$ and $\mathcal{G}'=\mathcal{G}'(A', M',$ $ N', B')$ be two generalized matrix algebra over a commutative ring $\mathcal{R}$. Suppose that $A'$ and $B'$ have only trivial idempotents, and at least one of $M'$ and $N'$ is nonzero. Let $\Omega\colon \mathcal{G}\longrightarrow \mathcal{G}'$ be an $\mathcal{R}$-linear mapping. Then $\Omega$ is an isomorphism if and only if $\Omega$ has one of the following forms:
\begin{enumerate}
\item[(1)] $
\Omega \left(\left[ \begin{array}{cc} a & m\\
n & b \end{array}\right]\right) = \left[\begin{array}{cc}\gamma
(a) & \gamma (a) m^\prime_0 -m^\prime _0 \delta (b) + \mu(m) \\ 
n^\prime_0 \gamma (a)
- \delta (b) n^\prime_0 + \nu(n) & \delta (b) \end{array}\right]
$, 
\end{enumerate}
where $\gamma\colon A\longrightarrow A^\prime$ and $\delta\colon B\longrightarrow B^\prime$ are two algebraic isomorphisms, 
$\mu\colon M\longrightarrow M^\prime$ is a $(\gamma, \delta)$-bimodule isomorphism, \ $\nu\colon N\longrightarrow 
N^\prime$ is a $(\delta, \gamma)$-bimodule isomorphism, $m^\prime_0 \in
M'$ and $n'_0 \in N'$ are two fixed elements.
\begin{enumerate} 
\item[(2)] $\Omega \left(\left[ \begin{array}{cc} a & m \\
n & b \end{array}\right]\right) = \left[ \begin{array}{cc}\sigma
(b) & m'_\ast \rho(a) - \sigma (b) m'_\ast + \tau (n) \\
\rho(a)n'_\ast - n'_\ast \sigma (b) + \zeta (m) & \rho(a)
\end{array} \right]$,
\end{enumerate}
where $\rho\colon A
\longrightarrow B'$ and $\sigma\colon B \longrightarrow A'$ are two algebraic isomorphisms, $\zeta\colon M\longrightarrow N'$ is a $(\rho, \sigma)$-bimodule isomorphism, $\tau\colon N\longrightarrow M'$ is a $(\sigma, \rho)$-bimodule isomorphism, $m'_\ast \in M'$ and $n'_\ast \in N'$ are two fixed elements.
\end{proposition}

Basing on Proposition \ref{xxsec5.4} we ask

\begin{question}\label{xxsec5.5}
Let $\mathcal{G}=\mathcal{G}(A, M, N, B)$ be a generalized matrix algebra over a commutative ring $\mathcal{R}$. Suppose that $A$ and $B$ have only trivial idempotents. Let $F_1, F_2, G_1, G_2$ be mappings from $\mathcal{G}$ into itself and $\Omega_1, \Omega_2, \Gamma_1, \Gamma_2$ be automorphisms from $\mathcal{G}$ into itself such that  
$$
F_1(x)\Omega_2(y)+F_2(y)\Omega_1(x)+\Gamma_1(x)G_2(y)+\Gamma_2(y)G_1(x)=0
$$
for all $x,y\in \mathcal{G}$. How about are the forms of $F_1, F_2, G_1, G_2$ and those of $\Omega_1, \Omega_2, \Gamma_1, \Gamma_2$ satisfying the above condition ?
\end{question}

In addition, some researchers extend the result about Lie isomorphisms between nest algebras on Hilbert spaces by Marcoux and
Sourour \cite{MarcouxSourour2} to the Banach space case, see \cite{QiHou2, QiHouDeng, WangLu}. We would liek to mention one recent remarkable result obtained by Qi, Hou and Deng. 

\begin{theorem}{\rm(\cite[Theorem 1.1]{QiHouDeng})}\label{xxsec5.6}
Let $\mathcal{N}$ and $\mathcal{M}$ be nests on Banach spaces $X$ and $Y$ over the (real or complex)
field $\mathbb{F}$ and let ${\rm Alg}\mathcal{N}$ and ${\rm Alg}\mathcal{M}$ be the associated nest algebras, respectively. Then a mapping $\Phi\colon  {\rm Alg}\mathcal{N}\longrightarrow {\rm Alg}\mathcal{M}$ is a Lie ring isomorphism, that is, $\Phi$ is additive, bijective and satisfies $\Phi([A, B])=[\Phi(A), \Phi(B)]$ for all $A, B\in {\rm Alg}\mathcal{N}$, if and only if $\Phi$ has the form $\Phi(A)=\Psi(A)+h(A)I$ for all $A\in {\rm Alg}\mathcal{N}$, where $\Psi$ is a ring isomorphism or the negative of a ring anti-isomorphism between the nest algebras and $h$ is an additive functional satisfying $h([A, B])=0$ for all $A, B\in {\rm Alg}\mathcal{N}$.
\end{theorem}

In light of Theorem \ref{xxsec5.6}, it is deserved to pay much more attention to centralizing traces and Lie triple isomorphisms of nest algebras on Banach spaces.

\begin{question}\label{xxsec5.7}
Let $\textbf{\rm X}$ be a Banach space, $\mathcal{N}$ be a nest of $\textbf{\rm X}$ and ${\mathcal Alg}(\mathcal{N})$ be the nest algebra associated with $\mathcal{N}$. Suppose that ${\mathfrak q}\colon {\mathcal Alg}(\mathcal{N})\times
{\mathcal Alg}(\mathcal{N})\longrightarrow {\mathcal Alg}(\mathcal{N})$ is an $\mathcal{R}$-bilinear
mapping. Then every centralizing trace ${\mathfrak T}_{\mathfrak q}:
{\mathcal Alg}(\mathcal{N})\longrightarrow {\mathcal Alg}(\mathcal{N})$ of ${\mathfrak q}$ is
proper.
\end{question}

Let us end this paper with the following conjecture.

\begin{conjecture}\label{xxsec5.8}
Let $\mathcal{N}$ and $\mathcal{M}$ be nests on Banach spaces $X$ and $Y$ over the (real or complex)
field $\mathbb{F}$ and let ${\rm Alg}\mathcal{N}$ and ${\rm Alg}\mathcal{M}$ be the associated nest algebras, respectively. Then a mapping $\Phi\colon  {\rm Alg}\mathcal{N}\longrightarrow {\rm Alg}\mathcal{M}$ is a Lie triple ring isomorphism, that is, $\Phi$ is additive, bijective and satisfies $\Phi([[A, B], C])=[[\Phi(A), \Phi(B)], \Phi(C)]$ for all $A, B, C\in {\rm Alg}\mathcal{N}$, if and only if $\Phi$ has the form $\Phi(A)=\Psi(A)+h(A)I$ for all $A\in {\rm Alg}\mathcal{N}$, where $\Psi$ is a ring isomorphism or the negative of a ring anti-isomorphism between the nest algebras and $h$ is an additive functional satisfying $h([[A, B],C])=0$ for all $A, B, C\in {\rm Alg}\mathcal{N}$.
\end{conjecture}

\vspace{4mm}
{\bf Acknowledgements}
\vspace{1mm}

This research was done when the first author visited the School of Mathematics
and Statistics at Beijing Institute of Technology in
the summer of 2014. She takes this opportunity to express her sincere thanks
to the School of Mathematics and Statistics and the Office of International
Affairs at Beijing Institute of Technology for the hospitality
extended to them during her visit. We are deeply grateful to a special
Training Program of International Exchange and Cooperation
of the Beijing Institute of Technology. All authors are thankful
to Professor Dominik Benkovi\v{c} and Professor Daniel Eremita for their
careful reading and constructive comments for the original manscript.

\end{document}